\documentclass[11pt,onecolumn]{article}

\usepackage{geometry}
\usepackage{amsfonts,amsmath,euscript} 
\usepackage{multirow}
\usepackage{multicol}        
\usepackage{amsthm}
\usepackage{algorithmic}
\usepackage{algorithm}
\usepackage{graphicx}        
\usepackage{color}
\usepackage{subfig}

\newtheorem{theorem}{\bf Theorem}[section]

\newtheorem{prop}{\bf Proposition}[section]

\newtheorem{example}{\bf Example}[section]

\newcommand{\bmat}{\left[ \begin{matrix}}
\newcommand{\emat}{\end{matrix} \right]}

\newcommand{\Tr} {\mbox{\rm tr}}

\newcommand{\Circ}{\mathop{\rm Circ}}
\newcommand{\tp}{^{\top}}

\newcommand{\bea}{\begin{eqnarray}}
\newcommand{\eea}{\end{eqnarray}}

\newcommand{\bsea}{\begin{subeqnarray}}
\newcommand{\esea}{\end{subeqnarray}}

\newcommand{\Symmetric}{{\mathfrak S}}

\newcommand{\Gc}{{\mathcal G}}
\newcommand{\Ic}{{\mathcal I}}

\newcommand{\ct}{\tilde{c}}
\def\Jb{\bar{J}}

\newcommand{\Ebb}{{\mathbb E}\,}

\newcommand{\Rbb}{\mathbb R}

\newcommand{\yb}{\mathbf  y}

\newcommand{\Ab}{\mathbf A}
\newcommand{\Bb}{\mathbf B}
\newcommand{\Cb}{\mathbf C}

\newcommand{\Rb}{\mathbf R}

\newcommand{\Tb}{\mathbf T}
\newcommand{\Ub}{\mathbf U}
\newcommand{\Vb}{\mathbf V}

\newcommand{\Sigmab}{\boldsymbol{\Sigma}}

\newcommand{\Psib}{\boldsymbol{\Psi}}
\newcommand{\Omegab}{\boldsymbol{\Omega}}

\begin{document}

\date{December 28, 2012}
\title{An Efficient Algorithm for Maximum-Entropy Extension  of Block--Circulant Covariance Matrices}


\newcommand{\footremember}[2]{%
   \footnote{#2}
    \newcounter{#1}
    \setcounter{#1}{\value{footnote}}%
}
\newcommand{\footrecall}[1]{%
    \footnotemark[\value{#1}]%
} 
\title{An Efficient Algorithm for Maximum-Entropy Extension  of Block--Circulant Covariance Matrices}
\author{%
    Francesca P. Carli \footremember{DEI}{Department of Information Engineering, University of Padova, Italy. 
    {\tt\small  carlifra@dei.unipd.it}, {\tt\small  augusto@dei.unipd.it},  {\tt\small  picci@dei.unipd.it}}%
    \and Augusto Ferrante\footrecall{DEI}%
    \and Michele Pavon\footremember{math}{Department of Pure and Applied Mathematics, University of Padova, Italy. {\tt\small pavon@math.unipd.it}} %
    \and Giorgio Picci\footrecall{DEI} 
    \footnote{Work partially supported by the Italian Ministry for Education and Resarch (MIUR) under PRIN grant ``Identification and Adaptive Control of Industrial Systems".}%
}

\maketitle

\begin{abstract}
{This paper deals with maximum entropy completion of partially specified block--circulant matrices. 
Since positive definite symmetric circulants happen to be covariance matrices of stationary periodic processes, in particular of stationary { reciprocal processes},  this problem has applications in signal processing, in particular to image modeling.
In fact it is strictly related to maximum likelihood estimation of bilateral AR--type representations of acausal signals subject to certain conditional independence constraints.
The maximum entropy completion problem for block--circulant matrices 
has recently been solved by the authors, 
although leaving open the problem of an efficient computation of the solution. 
In this paper, we provide an effcient algorithm for computing its solution which compares very favourably with existing algorithms designed for positive definite matrix extension problems. 
The proposed algorithm benefits from the analysis of the relationship between our problem and the 
band--extension problem for block--Toeplitz matrices also developed in this paper.}
\end{abstract}

%
%
%
%



\section{Introduction} \label{sec:intoduction}
We consider the problem of completing a partially specified block--circulant matrix under the constraint that the completed matrix should be  positive definite and block-circulant  with an inverse of banded structure. As shown in  \cite{CFPP-2011}, a block--circulant completion problem of this kind is a crucial tool for the  identification of a class of reciprocal processes. 
These processes (\cite{Jamison-70}, \cite{Levy-F-K-90}, \cite{Sand-94}) are a generalization of Markov processes which are particularly useful for modeling  random signals which live in a finite region of time or of the space line, for example  images. 
In this paper we consider   stationary reciprocal processes for which we refer the reader to \cite{LEvy-Ferr-SRP,Ferrante:1998:CFS} and references therein. In particular, stationary reciprocal processes of the autoregressive type can be described by linear models involving 
 a banded block--circulant concentration matrix\footnote{i.e. the inverse covariance matrix, also known as the precision matrix. } whose blocks are the (matrix--valued) parameters of the model. 

This problem fits in the general framework of  covariance extension problems introduced by A. P. Dempster \cite{Dempster-72} 
and studied by many authors (see 
{\cite{Burg1967}}, \cite{Dym-G-81}, \cite{GJSW-84}, \cite{DMS-89}, \ \cite{SK-86}, \cite{GKW-89}, \cite{BJL-89}, 
{\cite{Barrett1993}, \cite{Barrett1996}, \cite{Johnson1996}}, \cite{Gohberg-G-K-94}, 
{\cite{Laurent1997}, \cite{Johnson1998}, \cite{Glunt1999}, \cite{Laurent2001}, \cite{Dahl-V-R-08}}, \cite{FP-2011} and references therein). 
A key discovery by Dempster is that  the  inverse of the maximum entropy completion of a partially assigned covariance matrix 
has zeros exactly in the positions corresponding to the unspecified entries in the  given   matrix, 
a property which, from now on, will be referred to as the \emph{Dempster property} 
 (an alternative, concise proof of this statement can for example be found in  \cite{CG-2011}). 

{A relevant fact is that, even \emph{when the constraint of a circulant structure is imposed}, 
the inverse of the maximum entropy completion maintains the Dempster property. 
This fact has been first noticed in \cite{CFPP-2011} for a banded structure  
and then proved in complete generality, i.e. for arbitrary given elements within a block--circulant structure, in \cite{CG-2011}.   
Otherwise stated, this means that the solution of the Maximum Entropy block--Circulant Extension Problem (CME) and of the Dempster Maximum Entropy Extension Problem (DME) with  data consistent with a block--circulant structure, coincide. 
Note that this property does not hold, for example,  
for arbitrary missing elements in a block--Toeplitz structure: if we ask the completion to be Toeplitz, the maximum entropy extension fails to satisfy the Dempster property  
unless the given data lie on consecutive bands centered along the main diagonal 
(see \cite{Dym-G-81} and \cite{Gohberg-G-K-94} for a general  formulation of  matrix  extension  problems in terms of so--called {banded--algebra} techniques and for a thorough discussion of the so--called {band--extension} problem for block--Toeplitz matrices). Moreover, the block--Toeplitz band extension problem can be solved by factorization techniques and is essentially a linear problem. This is unfortunately no longer true when a block--circulant structure is imposed \cite{Carli-P-10} to the extension.} 

{The main contribution of this paper is to
  propose a new algorithm for solving the CME problem. A straightforward application of standard optimization algorithms would be  too expensive for large sized problems  like those we have in mind for, say applications to image processing. 
  Here we propose a new procedure which rests on duality theory and exploits information
on the structure of the problem as well as the circulant structure for computing the
solution of the CME. 
\\Since the solutions of the CME and of the DME with circulant--compatible data coincide, methods available in the literature for the DME can, in principle, be employed to compute the solution of CME.
In this respect, it has been shown that if the graph associated with the specified entries is chordal (\cite{Golumbic-80}), 
the solution of the DME can be expressed in closed form in terms of the principal minors of the 
covariance matrix, see \cite{BJL-89},  \cite{FKMN-00}, \cite{NFFKM-03}. 
In our problem however the sparsity pattern associated with the given entries  is  not chordal and  the maximum entropy completion has to be computed numerically. 
A number of specialized algorithms have been proposed in the graphical models literature; see \cite{Dempster-72,SK-86,Wermuth-S-77,Kullback-68}. These algorithms deal with the general unstructured  setting of Dempster and are not especially tailored to the circulant structure. 
A detailed comparison of our procedure with the best algorithms available so far is presented in Section \ref{sec:alg_cov_sel}. 
We show that the proposed algorithm outperforms the algorithms proposed in the graphical models literature for the solution of the DME, being especially suitable to deal with very large sized instances of the problem. 
} 
    
We shall first relate our work to  the solution of the band extension problem for  block--Toeplitz matrices  and show that the maximum entropy  circulant extension approximates arbitrarily closely the block--Toeplitz band extension with the same starting data, when the dimension of the circulant extension becomes large.  This result is in the spirit of 
the relation between stationary Markov and reciprocal processes on an infinite interval established by Levy in \cite{Levy-92} 
{and will be useful to provide an efficient initialization for the proposed algorithm. }
In this context, we shall briefly touch upon feasibility of the CME problem. 
The feasibility problem for generic blocks size and bandwidth has been addressed
in \cite{CFPP-2011} and \cite{CG-2011}, where 
a \emph{sufficient} condition on the data for a positive definite \emph{block}--circulant completion to exist 
has been derived. 
Here we shall derive a \emph{necessary and sufficient} condition for feasibility of the CME problem valid 
for the scalar case with unitary bandwidth.  

The outline of the paper is as follows.  
In Section \ref{sec:notation} we introduce some notation and state the entropy maximization problem. 
In Section \ref{sec:Markov_vs_Rec} the relation between the maximum entropy extension for banded Toeplitz and banded circulant matrices is investigated. A necessary an sufficient condition for feasibility is also derived in this Section.  
In Section \ref{sec:MGD}  
we describe 
the proposed procedure for the solution of the CME problem. 
Section \ref{sec:alg_cov_sel} contains a brief review and discussion of some of the most popular methods for the solution of the DME. 
A comparison of the proposed algorithm and the methods available in the literature is presented in Section \ref{sec:num_exp}. 
Section \ref{sec:conclusion} concludes the paper.

\section{Notation and preliminaries} \label{sec:notation}

All random variables  in this paper, denoted by boldface characters, have zero mean and finite second order moments. It is shown in \cite{CFPP-2011} that a 
wide--sense stationary 
$\mathbb{R}^m$--valued  process $\yb$ is stationary 
on 
{$\{0,\, 1,\, \dots ,\,N\}$} if and only if its covariance matrix, say $\Sigmab_N$, has a block--circulant symmetric structure, i.e. $\Sigmab_N$ is of the form 
$$ 
\Sigmab_N = \bmat 
{\Sigma}_0 &{\Sigma}_1^{\top} &\ldots        & {\Sigma}_{\tau}^{\top}&\ldots             & {\Sigma}_{\tau} & \ldots &       {\Sigma}_1\\
{\Sigma}_1 & {\Sigma}_0         &{\Sigma}_1^{\top} &\ddots           & {\Sigma}_{\tau}^{\top}& \ldots       &    \ddots     &    \vdots  \\
	 \vdots   &          &       \ddots    	&	\ddots&  &  \ddots  &              &  {\Sigma}_{\tau}  \\
{\Sigma}_{\tau}& \ldots &    {\Sigma}_1   & {\Sigma}_0 &  {\Sigma}_1 ^{\top}   &  \ldots &        \ddots & \\
                   \vdots & {\Sigma}_{\tau} &      \ldots      & & {\Sigma}_0 &          & \ldots &    {\Sigma}_{\tau} ^{\top} \\                                         {\Sigma}_{\tau}^{\top}   &           &       \ddots    &        & &          &    				  & \vdots   \\
 \vdots   &  \ddots  &      &            \ddots       &       &	\ddots	&  \ddots & {\Sigma}_1^{\top}   \\
{\Sigma}_1^{\top} &\ldots  & {\Sigma}_{\tau}^{\top}&\ldots  & {\Sigma}_{\tau} & & {\Sigma}_1  & {\Sigma}_0 \emat 
$$
where the $k$--th block, $\Sigma_k$, is given by $\Sigma_k = \Ebb \yb(t+k)\yb(t)^{\top}$.
We refer the reader to \cite{Davis-79} for an introduction to circulants; an extension of some relevant results 
for the \emph{block}--case can be found, for example, in \cite{CG-2011}. 
Here we just recall that the class of block--circulants is closed under sum, product, inverse and transpose. 
Moreover, all block--circulants 
are simultaneously diagonalized by the Fourier 
{block--}matrix of suitable size ({see \eqref{eqn:diag_circ}--\eqref{eqn:Psi_N} below}). 
 
The {\em differential entropy} $H(p)$ of a probability density function $p$ on $\Rbb^n$ is defined by
\begin{equation}\label{DiffEntropy}
H(p)=-\int_{\Rbb^n}\log (p(x))p(x)dx.
\end{equation}
In  case of a zero-mean Gaussian distribution $p$ with covariance matrix $\Sigmab_N$, it results 
\begin{equation}\label{gaussianentropy}
 H(p)=\frac{1}{2}\log(\det\Sigmab_N)+\frac{1}{2}n\left(1+\log(2\pi)\right).
\end{equation}
Let $\Symmetric_{N}$ denote the vector space of real {\em symmetric} matrices with $N\times N$ square blocks of dimension $m\times m$. Moreover, let ${\Ub}_N $ denote the block--circulant shift matrix with $N\times N$ blocks,
$$ 
{\Ub}_N =\bmat
0&I_m&0&\dots&0\\
0&0&I_m&\dots&0\\
\vdots&\vdots&
&\ddots&\vdots\\
       0&0&0&\dots&I_m\\
       I_m&0&0&\dots&0
       \emat \,, 
$$
$E_n$ the $N\times (n+1)$ block matrix
$$
E_n=
\bmat 
I_m & 0 & \dots & 0 \\
0 & I_m & & 0\\ 
\vdots &  & \ddots & \vdots \\
0& \dots & \dots & I_m\\
0& \dots&  & 0\\
\vdots & \ddots & & \vdots\\
0& & \dots & 0  \emat.
$$
and $\Tb_n \in \Symmetric_{n+1}$ 
the block--Toeplitz matrix  made of the first $n+1$, $m\times m$ covariance lags $\{\Sigma_0, \Sigma_1, \dots, \Sigma_n\}$, 
\begin{equation}\label{ToeplInitData}
\Tb_n:=\bmat 
\Sigma_{0}&\Sigma_{1}\tp&\dots & \dots &\Sigma_{n}\tp \\
\Sigma_{1}& \Sigma_0 & \Sigma_1^\top &  & \vdots\\
\vdots & \ddots & \ddots & \ddots & \vdots \\
\vdots &  & \ddots & \ddots & \Sigma_1^\top \\
\Sigma_{n}& \dots & \dots & \Sigma_1 &\Sigma_{0} \emat\,.
\end{equation}
The symmetric block--Toeplitz matrix $\Tb_n$ is completely specified by its first block--row, so, with obvious notation, it will be also denoted as 
$$\Tb_n = \text{\rm Toepl} \left(\Sigma_0, \Sigma_1^\top, \dots, \Sigma_n^\top\right).$$

The maximum entropy covariance extension problem for block--circulant matrices (CME) can be stated as follows. 
\begin{subequations}\label{MaxEntProbl}
\begin{eqnarray}
&&\max \left\{\log\det\Sigmab_N \mid \Sigmab_N \in\Symmetric_{N},\; \Sigmab_N >0\right\}\\&&{\rm subject \;to:}\nonumber\\&&E_n^{\top} \Sigmab_N  E_n=\Tb_n,\label{c1}\\&&{\Ub}_N \tp \Sigmab_N {\Ub}_N =\Sigmab_N. \label{c2}
\end{eqnarray}
\end{subequations}
where we have exploited the fact that a matrix $\Cb_{N}$ with $N \times N$ blocks is block--circulant if and only if 
it commutes with $ {\Ub}_N$, namely if and ony if $\Cb_{N} = {\Ub}_N \tp \Cb_{N}{\Ub}_N$. 
Problem \eqref{MaxEntProbl} is a \emph{convex optimization problem} since we are minimizing a strictly convex function on the intersection of a convex cone (minus the zero matrix) with a linear manifold. 
If we do not impose the completion to be block--circulant, we obtain 
the 
covariance selection problem studied by A. P. Dempster (DME) in \cite{Dempster-72}. 

Notice that, although in Problem \ref{MaxEntProbl} we are maximizing the entropy functional over zero--mean \emph{Gaussian} densities, 
we are {\em not} actually restricting ourselves to the case of Gaussian distributions. 
Indeed, the Gaussian distribution with (zero mean and) covariance matrix solving \eqref{MaxEntProbl} 
maximizes the entropy functional \eqref{DiffEntropy} over the larger family of  (zero mean) probability densities whose covariance matrix satisfies the boundary conditions \eqref{c1}, \eqref{c2}, see \cite[Theorem 7.2]{CFPP-2011}.



\section{Relation with the block-Toeplitz covariance extension problem }
\label{sec:Markov_vs_Rec}

In this Section, we shall point out a relation between the solutions of the maximum entropy band extension problem for block--circulant and block--Toeplitz matrices.  
\\
{Let $\Tb_n$ and $E_n$ be the block matrices defined in Section \ref{sec:notation}. Moreover let $\Ab_{N-1}$ and $\Bb_{N-1}$ be the $(N-1)\times N$ block shift matrices 
$$
\Ab_{N-1}= \bmat 
I_m & 0 & 0 & \dots & 0 \;&0 \,\\
0 & I_m & 0 & & 0\; & 0 \,\\
0 & 0 & I_m & & 0 \;& 0\\
\vdots &  &  & \ddots &  \vdots\; & \vdots \\
0 & 0 & 0 & \dots & I_m & 0
\emat \,, \;\;
\Bb_{N-1} = \bmat
0 & I_m & 0 & 0 & \dots & 0 \\
0 & 0& I_m & 0 &  & 0\\
0 & 0 & 0 & I_m & \dots & 0 \\
\vdots &  &  &  & \ddots & \vdots \\
0 & 0 & 0 & 0& \dots & I_m 
 \emat
$$
The maximum entropy band extension problem for block--Toeplitz matrices (TME) can be stated  as follows.
\begin{subequations}\label{TMEP}
\begin{eqnarray}
&&\max \left\{\log\det\Sigmab_N \mid \Sigmab_N \in\Symmetric_{N},\; \Sigmab_N >0\right\}
\\&&{\rm subject \;to:}\nonumber
\\&&E_n^{\top} \Sigmab_N  E_n=\Tb_n,\label{c1_TMEP}
\\&&{\Ab}_{N-1} \Sigmab_N {\Ab}_{N-1}^\top = \Bb_{N-1} \Sigmab_N \Bb_{N-1}^\top. \label{c2_TMEP}
\end{eqnarray}
\end{subequations}}
This  problem   has a long history, and was probably the first 
matrix completion problem studied in the literature (\cite{Dym-G-81}, \cite{Gohberg-G-K-94}). 
As mentioned in the Introduction, it can be solved by factorization techniques, in fact, by the celebrated Levinson--Whittle algorithm  \cite{Whittle-63}  and is essentially a linear problem. Below we shall show that  for $N\to \infty$, the solution of the CME problem can be approximated arbitrarly closely  in terms of the  solution of the Toeplitz band extension problem. 
The Theorem reads as follows. 

\begin{theorem} \label{thm:link_TMEExt_CMEExt}
{\em  Let $\Tb_n$ be positive definite and 
let $\{\hat{\Sigma}_{k},\,k=n+1,n+2,\ldots\}$ 
be the maximum entropy 
block--\emph{Toeplitz} extension of $\{\Sigma_0, \Sigma_1, \dots, \Sigma_n\}$ solution of the  
TME problem \ref{TMEP}. 
Then, for $N$ large enough, the symmetric block--circulant matrix $\Sigmab_N^{(c)}$ given by   
\begin{equation}\label{eqn:compl_circ_ME_da_Toepl_even}
{\rm Toepl} \left({\Sigma}_0, {\Sigma}_1\tp, \dots, {\Sigma}_n\tp, \hat{\Sigma}_{n+1}\tp, \dots, \hat{\Sigma}_{\frac{N}{2}-1}\tp ,\hat{\Sigma}_{\frac{N}{2}}\tp+\hat{\Sigma}_{\frac{N}{2}},\hat{\Sigma}_{\frac{N}{2}-1},\dots,\hat{\Sigma}_{n+1}, {\Sigma}_n, \dots {\Sigma}_1\right),
\end{equation} 
for $N$ even, and 
\begin{equation}\label{eqn:compl_circ_ME_da_Toepl_odd}
{\rm Toepl} \left({\Sigma}_0, {\Sigma}_1\tp, \dots, {\Sigma}_n\tp,\hat{\Sigma}_{n+1}\tp, \dots, \hat{\Sigma}_{\frac{N-1}{2}}\tp,\hat{\Sigma}_{\frac{N-1}{2}}, \dots,  \hat{\Sigma}_{n+1},{\Sigma}_{n},\dots, {\Sigma}_1\right), 
\end{equation}
for $N$ odd, is a covariance matrix which for $N\to \infty$  is arbitrarily close  to the  $mN\times mN$  maximum entropy  block--\emph{circulant} extension of $\Tb_n$ solution of the CME \ref{MaxEntProbl}.}
\end{theorem}

\begin{proof}
That $\Sigmab_N^{(c)}$ is a valid covariance matrix for $N$ large enough follows from \cite[Theorem 5.1]{CFPP-2011}. 
{It remains to show that $\Sigmab_N^{(c)}$ given by \eqref{eqn:compl_circ_ME_da_Toepl_even}, \eqref{eqn:compl_circ_ME_da_Toepl_odd} tends to the maximum entropy block--circulant extension of $\Tb_n$, say $\Sigmab_N^o$, i.e. that 
\begin{equation}\label{eqn:lim}
\lim_{N \rightarrow \infty} \left\| \Sigmab_N^{(c)} - \Sigmab_N^o\right\| = 0\,.
\end{equation}
To this aim, we recall that the maximum entropy completion $\Sigmab_N^{(c)}$ is the unique completion of the given data whose inverse has the property to be zero in the complementary positions of those assigned (\cite{CFPP-2011}, \cite{CG-2011}, \cite{Dempster-72}). 
Thus, \eqref{eqn:lim} holds if and only if for $N\to \infty$ the inverse of $\Sigmab_N^{(c)}$ tends to be banded block--circulant 
$$
\lim_{N \rightarrow \infty} \left\| \left(\Sigmab_N^{(c)}\right)^{-1} - \left(\Sigmab_N^o\right)^{-1}\right\| = 0 \,,
$$
i.e. if and only if its off--diagonal blocks tend uniformly to zero (faster than $N^2$).}

{To show this,} recall that 
$\Sigmab_N^{(c)}$ can be {block}--diagonalized as 
\begin{equation}\label{eqn:diag_circ}
\Sigmab_N^{(c)} = \Vb \Psib_N \Vb^\ast
\end{equation}
where 
$\Vb$ is the Fourier block-matrix whose $(k,l)$-th block is 
\begin{equation}\label{eqn:Fourier_matrix}
V_{kl}:=1/\sqrt{N} \exp\left[-{ {\rm j}}2\pi(k-1)(l-1)/N\right]I_m
\end{equation}
and $\Psib_N$ is the block--diagonal matrix 
\begin{equation}\label{eqn:Psi_N}
\Psib_N:={\rm diag}\left(\Psi_0,\Psi_1,\dots,\Psi_{N-1}\right),
\end{equation}
whose diagonal blocks $\Psi_\ell$, are the coefficients of the finite Fourier transform of the first block row of $\Sigmab^{(c)}_N$  
\begin{equation}\label{eqn:Psi_ell}
\Psi_\ell=\hat{\Sigma}_0+{\rm e}^{{\rm j}\vartheta_\ell}\hat{\Sigma}_1\tp+\left({\rm e}^{{\rm j}\vartheta_\ell}\right)^2\hat{\Sigma}_2\tp+\dots +\left({\rm e}^{{\rm j}\vartheta_\ell}\right)^{N-2}\hat{\Sigma}_2+\left({\rm e}^{{\rm j}\vartheta_\ell}\right)^{N-1}\hat{\Sigma}_1,
\end{equation}
with $\vartheta_\ell :=- 2\pi\ell/N$. 
Thus in particular 
$$
\left(\Sigmab_N^{(c)}\right)^{-1} = \Vb \Psib_N^{-1} \Vb^\ast
$$
where 
$$
\Psib_N^{-1}:={\rm diag}\left(\Psi_0^{-1},\Psi_1^{-1},\dots,\Psi_{N-1}^{-1}\right).
$$
Now, 
{let us consider the block--Toeplitz band extension of the given data $\Tb_n$, $\{\hat{\Sigma}_{k},\,k=0,1,2,\ldots\}$, and the associated spectral density matrix}  
\begin{equation}\label{eqn:spectral_density}
\Phi(z) := \hat{\Sigma}_0 + \sum_{i=1}^\infty \hat{\Sigma}_i z^{-i} + \left(\sum_{i=1}^\infty \hat{\Sigma}_i z^{-i}\right)^*\,.
\end{equation}
It is well--known \cite{Whittle-63} that $\Phi(z)$ can be expressed in factored form as 
\begin{equation}\label{eqn:mult_dec_Levinson}
\Phi (z) =  \left[L_n(z^{-1})\right]^{-1} \Lambda_n \left[L_n(z^{-1})\right]^{-\ast} 
\end{equation}
where $L_n(z^{-1})$ is the $n$--th Levinson--Whittle matrix polynomial associated with the block--Toeplitz matrix $\Tb_n$ 
\begin{equation}\label{eqn:Levinson_matrix_polynomial}
L_n(z^{-1}) = \sum_{k=0}^n A_n(k) z^{-k} 
\end{equation}
with the $A_n(k)$'s and $\Lambda_n=\Lambda_n^\top >0$ 
{being the} solutions of the Yule--Walker type equation  
\begin{equation}\label{eqn:Yule-Walker_eqn_for_A}
\bmat A_n(0) & A_n(1) & \dots & A_n(n)\emat 
\Tb_n^\top = 
\bmat \Lambda_n & 0 & \dots & 0\emat \,.
\end{equation}
Note that $\Phi(z)^{-1}=L_n(z^{-1})^* \Lambda_n^{-1} L_n(z)$ is a Laurent polynomial, that can be written as  
{\begin{align*}
\Phi(z)^{-1} 
&=M_0 + \left(M_1 z + M_2 z^2 + \dots + M_n z^n \right)+ \left(M_1 z + M_2 z^2 + \dots + M_n z^n \right)^*
\end{align*}}
Moreover, 
the $\Psi_\ell$'s 
{in \eqref{eqn:Psi_ell}} can be written as 
\footnote{{For $N$ even ${\rm e}^{{\rm j}\vartheta_\ell h}={\rm e}^{{\rm -j}\vartheta_\ell h} = -1$, 
so that $\left({\rm e}^{{\rm j}\vartheta_\ell}\right)^h \hat{\Sigma}_h^\top +\left({\rm e}^{{\rm -j}\vartheta_\ell}\right)^h \hat{\Sigma}_h= - \left(\hat{\Sigma}_h+\hat{\Sigma}_h^\top\right)$.} }
\begin{equation}\label{eqn:Psi_ell_simm}
\Psi_\ell= \hat{\Sigma}_0 + {\rm e}^{{\rm j}\vartheta_\ell} \hat{\Sigma}_1^\top + \dots + \left({\rm e}^{{\rm j}\vartheta_\ell}\right)^h  \hat{\Sigma}_h^\top + 
{\rm e}^{{\rm -j}\vartheta_\ell} \hat{\Sigma}_1 + \dots +   \left({\rm e}^{{\rm -j}\vartheta_\ell}\right)^h  \hat{\Sigma}_h 
\end{equation} 
where 
$$
h:=\left\{\begin{array}{ll} 
\frac{N-1}{2},& N {\rm\ odd}\\
N/2, & N {\rm\ even}
\end{array}\right.
$$
{
Now, comparing expression \eqref{eqn:Psi_ell_simm} with \eqref{eqn:spectral_density}, we can write}
\begin{equation}\label{feas-fppsi}
\Psi_\ell=\Phi\left({\rm e}^{{\rm j}\vartheta_\ell}\right)-\left[\Delta\Phi_N \left({\rm e}^{{\rm j}\vartheta_\ell}\right)+\Delta\Phi_N^\ast \left({\rm e}^{{\rm j}\vartheta_\ell}\right)\right]
\end{equation}
where 
$$
\Delta\Phi_N(z)  := \sum_{i=h+1}^\infty \hat{\Sigma}_iz^{-i}\,. 
$$
Since the causal part of $\Phi(z)$ is a rational function with poles inside the unit circle,  
$$\sup_{l=0,\, \dots,\, N-1}\big\|\Delta\Phi_N \left({\rm e}^{{\rm j}\vartheta_\ell}\right)+\Delta\Phi_N^\ast \left({\rm e}^{{\rm j}\vartheta_\ell}\right)\big\|\rightarrow 0$$ exponentially fast for $N\rightarrow \infty$. 
It follows that, for $N \rightarrow \infty$, $\Psib_\ell^{-1}$ tends to $\left(\Phi(e^{j\theta_\ell})\right)^{-1}$, which is given by 
\begin{align}\label{eqn:Psi_l_inv}
\left(\Phi(e^{j\theta_\ell})\right)^{-1}
& =M_0 + M_1 {\rm e}^{{\rm j}\vartheta_\ell} + \dots + M_n \left({\rm e}^{{\rm j}\vartheta_\ell}\right)^n \nonumber \\
& \; + M_1^\top \left({\rm e}^{{\rm j}\vartheta_\ell}\right)^{N-1}+ \dots + M_n^\top \left({\rm e}^{{\rm j}\vartheta_\ell}\right)^{N-n}  , 
\end{align}
{for every $ \ell=0, \, 1,\, \dots, \, N-1 $. 
In other words, $\Psib_\ell^{-1}$ tends to the finite Fourier transform of a sequence of the form 
$$
M_0, M_1^\top, M_2^\top, \dots, M_n^\top, 0 , \dots, 0 , M_n, \dots, M_1, 	
$$
i.e. $\left(\Sigmab_N^{(c)}\right)^{-1}$ tends to be banded block--circulant, as claimed. } 
\end{proof}

This result is very much in the spirit of the findings by Levy \cite{Levy-92}, which establish  a relation between stationary Markov and reciprocal processes on an infinite interval 
{and will be used in Section \ref{subsec:initialization} to provide an efficient inizialization for the proposed algorithm. }

{Feasibility of the CME problem has been addressed in \cite{CFPP-2011}, where a sufficient condition on the data for a positive definite block-circulant completion to exist has been derived. } 
There is a simple, yet, to the best of our knowledge, still unnoticed,  
necessary and sufficient condition for the existence of a positive definite circulant completion  
 for scalar (blocks of size $1 \times 1$) entries and bandwidth $n=1$ which
can be derived 
by combining the results in \cite{Barrett1993}, \cite{Dempster-72}, \cite{CG-2011}. It reads as follows.  
\begin{prop}\label{thm:feas_bs1_bw1} 
{Let $N \geq 4$. 
The partially specified circulant matrix 
\begin{equation}\label{eqn:part_spec_circ_b1_b1}
\bmat
\sigma_0 & \sigma_1 & ? & \dots & \dots & ? & \sigma_1 \\
\sigma_1 & \sigma_0 & \sigma_1 & ? & \dots & \dots & ? \\
? & \sigma_1 & \sigma_0 & \sigma_1 & ? &\dots & ? \\
\vdots & & & \ddots & & & \\
\vdots &  &  & & \ddots & &  ?\\
? & \dots & \dots & ? & \sigma_1 & \sigma_0 & \sigma_1 \\
\sigma_1 & ? & \dots & \dots & ? & \sigma_1 & \sigma_0
\emat
\end{equation}
admits a positive definite \emph{circulant} completion if and only if  $|\sigma_1| < \sigma_0$, for $N$ even, and 
$\cos\left(\frac{N-1}{N} \pi\right)\sigma_0  < \sigma_1 < \sigma_0$, for $N$ odd. }
\end{prop} 
\begin{proof}
In \cite[Corollary 5]{Barrett1993} it is shown that the partially specified circulant matrix \eqref{eqn:part_spec_circ_b1_b1} admits a positive definite (but not necessarily circulant) 
completion if and only if  $|\sigma_1| < \sigma_0$, for $N$ even, and 
$\cos\left(\frac{N-1}{N} \pi\right)\sigma_0  < \sigma_1 < \sigma_0$, for $N$ odd.
On the other hand, Dempster \cite{Dempster-72} shows that if there is any positive definite symmetric matrix
which agrees with the partially specified one in the given positions, then there exists exactly one such a 
matrix with the additional property that its inverse has zeros in the complementary positions of those specified 
and this same matrix is the one which maximizes the entropy functional among all the normal models whose covariance matrix agrees with the given data. But, accordingly to the findings in \cite{CG-2011}, 
if the given data are consistent with a circulant structure, the maximum entropy completion is necessarily circulant, 
which concludes the proof.   
\end{proof}

Proposition \ref{thm:feas_bs1_bw1} provides an explicit condition on the off--diagonal entries for the CME to be feasible. Moreover, for given $\sigma_0$ and $\sigma_1$, it states that feasibility depends on the size $N$ of the asked completion.  This confirms, by a completely independent argument, the findings in \cite[Theorem 5.1]{CFPP-2011}, where the dependency of feasibility on the completion size $N$ has been first noticed (and proved for a generic block--size and bandwidth). 
A clarifying example, which also makes use of the characterization of the set of the positive definite completions in \cite{CG-2011}, is presented in Appendix \ref{app:example}.

\section{A new algorithm for the solution of the CME problem}\label{sec:MGD}
 
In this Section we shall derive our new algorithm to solve the CME problem. 
{The derivation rests upon duality theory for the CME problem developed in \cite[Section VI]{CFPP-2011} and profits by 
the structure of our CME 
along with the properties of block--circulant matrices to devise a computationally advantageous procedure for the computation of its solution.  } 

Consider the CME as defined in \eqref{MaxEntProbl} and define the linear map 
\begin{equation}
\begin{array}{lccl}
A \;: \; & \Symmetric_{n+1}\times\Symmetric_N  & \rightarrow & \Symmetric_N\\
         & (\Lambda,\Theta)                    & \mapsto     & E_n\Lambda E_n\tp +{\Ub}_N \Theta {\Ub}_N \tp -\Theta
\end{array}
\end{equation}
and 
the set 
\begin{align}\label{eqn:L_+}
{\cal L}_+ :=&\{(\Lambda,\Theta)\in (\Symmetric_{n+1}\times\Symmetric_N)\mid (\Lambda,\Theta)\in(\ker(A))^\perp, \nonumber \\
           &  \left(E_{n}\Lambda E_n^{\top} +{\Ub}_N \Theta {\Ub}_N \tp -\Theta\right) > 0\}.  
\end{align}
${\cal L}_+$ is an open, convex subset of $(\ker(A))^\perp$. 
Letting $\langle A,\,B\rangle := \Tr AB\tp$, the {\em Lagrangian function} results 
\begin{align*}
L(\Sigmab_N,\Lambda,\Theta)&:=- \, \Tr\log\Sigmab_N+ \left\langle \Lambda, \left(E_n^{\top}\Sigmab_N E_n-\Tb_n\right)\right\rangle, + \left\langle \Theta, \left({\Ub}_N \tp \Sigmab_N {\Ub}_N -\Sigmab_N\right) \right\rangle \\
&\,=-\, \Tr\log\Sigmab_N+\Tr\left(E_n\Lambda E_n\tp \Sigmab_N\right) \, -\Tr\left(\Lambda\Tb_n\right)+\Tr\left({\Ub}_N \Theta {\Ub}_N \tp \Sigmab_N \right) \, -\Tr\left(\Theta \Sigmab_N \right)
\end{align*}
and its first variation (at $\Sigmab_N$ in direction $\delta\Sigmab_N \in \Symmetric_N$) is 
\begin{align*}
\delta L(\Sigmab_N,\Lambda,\Theta;\delta\Sigmab_N)=& -\Tr\left(\Sigmab_N^{-1}\delta\Sigmab_N\right)+\Tr\left(E_n\Lambda E_n\tp \delta \Sigmab_N\right)  + \Tr\left(\left({\Ub}_N \Theta {\Ub}_N \tp -\Theta\right)\delta\Sigmab_N\right).
\end{align*}
Thus $\delta L(\Sigmab_N,\Lambda,\Theta;\delta\Sigmab_N)=0,\; \forall \delta\Sigmab_N\in\Symmetric_N$ if and only if
$$
\Sigmab_N^{-1}=E_n\Lambda E_n\tp +{\Ub}_N \Theta {\Ub}_N \tp -\Theta.
$$
It follows that, for each fixed pair $(\Lambda,\Theta)\in {\cal L}_+$, the unique $\Sigmab_N^o$ minimizing the Lagrangian over ${\mathfrak S}_{N,+}:=\{\Sigmab_N \in\Symmetric_N,\;\Sigmab_N>0\}$ is 
\begin{equation}
\Sigmab_N^o=\left(E_n\Lambda E_n\tp +{\Ub}_N \Theta {\Ub}_N \tp -\Theta\right)^{-1}.
\end{equation}
Moreover, computing the Lagrangian at $\Sigmab_N = \Sigmab_N^o$ results in 
\begin{align}
\nonumber L(\Sigmab_N^o,\Lambda,\Theta) & = -\Tr\log\left(\left(E_{n}\Lambda E_n^{\top} + {\Ub}_N \Theta {\Ub}_N \tp -\Theta\right)^{-1}\right) \nonumber \\
& \;\; \; \; + \Tr \Big[\left(E_{n}\Lambda E_n^{\top} +{\Ub}_N \Theta {\Ub}_N \tp -\Theta\right) \nonumber \\ 
& \;\;\; \; \left(E_{n}\Lambda E_n^{\top} +{\Ub}_N \Theta {\Ub}_N \tp -\Theta\right)^{-1}\Big] -\Tr(\Lambda\Tb_n)\nonumber \\
& =\Tr\log\left(E_{n}\Lambda E_n^{\top} +{\Ub}_N \Theta {\Ub}_N \tp -\Theta\right) \nonumber \\ 
& \;\;\;\; + \Tr I_{mN}-\Tr\left(\Lambda\Tb_n\right).\nonumber
\end{align}
This is a strictly concave function on ${\cal L}_+$ whose maximization is the {\em dual problem} of (CME).
We can equivalently consider the convex problem
\begin{equation}\label{eqn:dual}
\min\left\{J(\Lambda,\Theta),(\Lambda,\Theta)\in{\cal L}_+\right\},
\end{equation}
where $J$ is given by
\begin{equation}\label{eqn:dual_function}
J(\Lambda,\Theta) = \Tr\left(\Lambda\Tb_n\right)-\Tr\log\left(E_n\Lambda E_n\tp +{\Ub}_N \Theta {\Ub}_N \tp -\Theta\right).
\end{equation}
It can be shown (\cite[Theorem 6.1]{CFPP-2011}) that the function $J$ admits a unique minimum point  $\left(\bar{\Lambda},\bar{\Theta}\right)$ in $\cal{L}_+$. 
The gradient of $J$ is 
\begin{subequations}
\begin{align}
\nabla_\Lambda J(\Lambda, \Theta) &= -E_n^\top \left(E_n \Lambda E_n^\top + {\Ub}_N \Theta {\Ub}_N^\top - \Theta \right)^{-1} E_n + \Tb_n  \label{eqn:grad_Lambda}\\
\nabla_\Theta J(\Lambda, \Theta) &= -{\Ub}_N^\top \left(E_n \Lambda E_n^\top + {\Ub}_N \Theta {\Ub}_N^\top - \Theta \right)^{-1} {\Ub}_N + \left(E_n \Lambda E_n^\top + {\Ub}_N \Theta {\Ub}_N^\top - \Theta \right)^{-1}  \label{eqn:grad_Theta}
\end{align}
\end{subequations}
Thus the application of whatever first--order iterative method for the minimization of $J$ 
would involve repeated inversions of the $mN\times mN$ block matrix 
$(E_n \Lambda E_n^\top +$ ${\Ub}_N \Theta {\Ub}_N^\top -\Theta)$, which could be a prohibitive task for $N$ large.  
Neverthless, by exploting our knowlwdge of the problem, we can devise 
the following alternative. 
Let $(\bar{\Lambda},\bar{\Theta})$ be the unique minimum point of the functional $J$ on ${\cal L}_+$. 
We know that $(\bar{\Lambda},\bar{\Theta})$ are such that $\Sigma^o = E_n\bar{\Lambda} E_n^\top +U_N \bar{\Theta} U_N ^\top -\bar{\Theta}$ is circulant. 
Thus, one can think of restricting the search for the solution of the optimization problem to the set  
\begin{equation}\label{eqn:set_circ_1}
\left\{ \left(\Lambda,\Theta\right) \, \mid \, 
\left(E_n \Lambda E_n^\top +\Ub_N \Theta \Ub_N ^\top -\Theta \right) \; \text{ is circulant}  \right\}\,.
\end{equation}
If we denote by $\mathfrak{C}_{N}$  the linear subspace of symmetric, block--circulant matrices and  by $\Pi_{\mathfrak{C}_{N}}$ the orthogonal projection on $\mathfrak{C}_{N}$, the set \eqref{eqn:set_circ_1} can be written as 
\begin{equation}\label{eqn:set_circ_2}
\left\{\left(\Lambda,\Theta\right) \, \mid \, 
\Pi_{\mathfrak{C}_{N}^\perp}\left(E_n \Lambda E_n^\top +\Ub_N \Theta \Ub_N ^\top -\Theta \right)=0  \right\}\,.
\end{equation}
We can now exploit the characterization of the matrices belonging to the orthogonal complement of  $\mathfrak{C}_{N}$ in \cite[Lemma 6.1]{CFPP-2011}, which states that a symmetric matrix $M$ belongs to the orthogonal complement of  $\mathfrak{C}_{N}$, say ${\mathfrak{C}_{N}}^\perp$, if and only if, for some $N\in\Symmetric_N$, it can be expressed as
$M={\Ub}_N N{\Ub}_N \tp -N$. 
Thus $(\Ub_N \Theta \Ub_N ^\top -\Theta) \in {\mathfrak{C}_{N}}^\perp$ and set \eqref{eqn:set_circ_2} can be written as
\begin{equation} \label{eqn:set_Lambda_Theta_st_map_is_circ}
\left\{\left(\Lambda,\Theta\right) \, \mid \, 
\Pi_{\mathfrak{C}_{N}^\perp}\left(E_n \Lambda E_n^\top\right) = - \left(\Ub_N \Theta \Ub_N ^\top -\Theta \right)  \right\}\,.
\end{equation}
If we compute the dual function $J$ on the set \eqref{eqn:set_Lambda_Theta_st_map_is_circ}, we obtain
\begin{align}
&J(\Lambda,\Theta)\mid_{\left\{\left(\Lambda,\Theta\right) \, \mid \, 
\Pi_{\mathfrak{C}_{N}^\perp}\left(E_n \Lambda E_n^\top\right) = - \left(U_N \Theta U_N ^\top -\Theta \right)\right\}}  \nonumber\\
&= \Tr\left(\Lambda \Tb_n\right)-\Tr \log \left(E_n \Lambda E_n^\top + U_N \Theta U_N ^\top -\Theta^\top \right) \nonumber \\
&= \Tr\left(\Lambda \Tb_n\right)-\Tr \log \left(E_n \Lambda E_n^\top - \Pi_{\mathfrak{C}_{N}^\perp} \left(E_n \Lambda E_n^\top \right) \right) \nonumber \\ 
&= \Tr\left(\Lambda \Tb_n\right)-\Tr \log \left(\Pi_{\mathfrak{C}_{N}} \left(E_n \Lambda E_n^\top  \right) \right)
\end{align}
where an explicit formula for the orthogonal projection of $E_n \Lambda E_n^\top$ on the subspace of symmetric, block--circulant matrices is given by Theorem 7.1 in \cite{CFPP-2011}. 
In fact, if we denote with 
$$
{\Lambda}=\bmat 
{\Lambda}_{00}     & {\Lambda}_{01}   & \ldots  & {\Lambda}_{0n}\\
{\Lambda}_{01}\tp  &{\Lambda}_{11}    & \ldots  & {\Lambda}_{1n}\\
           \vdots  &                  & \ddots  & \vdots\\
{\Lambda}_{0n}\tp  &{\Lambda}_{1n}\tp & \ldots  & {\Lambda}_{nn} 
\emat \,,
$$ 
it can be shown that 
the orthogonal projection of $E_n \Lambda E_n^\top$ onto $\mathfrak{C}_{N}$, say $\Pi_{{\Lambda}}$, is 
the banded block--circulant matrix given by 
$$
\Pi_{{\Lambda}}:=\Pi_{\mathfrak{C}_{N}}\left(E_n {\Lambda} E_n\tp \right)= 
\bmat
\Pi_0 & \Pi_1^\top & \dots & \Pi_n^\top & 0 & \dots & 0 & \Pi_n & \dots & \Pi_1 \\
\Pi_1 & \Pi_0 & \Pi_1^\top & \dots & \Pi_n^\top & 0& \dots & 0 & \ddots & \vdots \\
\vdots & \ddots & \ddots & \ddots & & \ddots & \ddots & & \ddots& \Pi_n\\
 \Pi_n& & \ddots & \ddots & \ddots & &  \ddots& \ddots & & 0\\
  0& \ddots & & \ddots & \ddots & \ddots & &  \ddots& \ddots & \vdots \\
  \vdots&  \ddots& \ddots & &  \ddots & \ddots & \ddots & & \ddots & 0 \\
  0& & \ddots & \ddots & & \ddots & \ddots & \ddots & & \Pi_n^\top \\
  \Pi_n^\top& \ddots& &  \ddots& \ddots & & \ddots & \ddots & \ddots & \vdots \\
  \dots & \Pi_n^\top & 0& \dots & 0 & \Pi_n & \dots & \Pi_1 & \Pi_0 & \Pi_1^\top \\
  \Pi_1^\top & \dots & \Pi_n^\top & 0 & \dots & 0 & \Pi_n & \dots & \Pi_1 &  \Pi_0 
\emat
$$
with 
\begin{subequations}\label{subeq:Pi_k}
\begin{eqnarray}
&&\Pi_{0}=\frac{1}{N}\,({\Lambda}_{00}+{\Lambda}_{11} +\ldots + {\Lambda}_{nn}), \\
&&\Pi_{1}=\frac{1}{N}({\Lambda}_{01} + {\Lambda}_{12} +\ldots + {\Lambda}_{n-1,n})\tp,\\
& & \vdots \nonumber\\
&&\Pi_{n}=\frac{1}{N} {\Lambda}_{0n}\tp \,,
\end{eqnarray}
\end{subequations}
and $\Pi_i=0$,  forall $i$ in the interval $\, n+1\leq i \leq N-n-1\,$.
Let us denote with $\bar{J}$ the restriction of $J$ on \eqref{eqn:set_Lambda_Theta_st_map_is_circ} 
\begin{equation} \label{eqn:pseudo_DualFunction}
\Jb (\Lambda) := \Tr\left(\Lambda \Tb_n\right) - \Tr \log \left\{\Pi_{\mathfrak{C}_{N}}\left(E_n \Lambda E_n^\top\right)\right\}  \,. 
\end{equation} 
The gradient of the modified functional $\Jb$ 
is 
\begin{equation}\label{eqn:grad_J_bar}
\nabla_\Lambda \Jb(\Lambda) = -E_n^\top \Pi_{{\Lambda}}^{-1} E_n + \Tb_n \,.
\end{equation}
Again, the computation of the gradient matrix involve the inversion of an $mN\times mN$ matrix, namely 
the projection on the subspace of symmetric block--circulant matrices of $E_n \Lambda E_n^\top$, $\Pi_{\Lambda}$. 
Neverthless, notice that this time the $mN\times mN$ matrix to be inverted is \emph{block--circulant}, which implies that its inverse can be efficiently computed by exploting the block--diagonalization 
\begin{equation}\label{eqn:blk_diag_circ_grad}
\Pi_{\Lambda}= \Vb \Omegab_N \Vb^\ast\,,
\end{equation}
where $\Vb$ is the block--Fourier matrix \eqref{eqn:Fourier_matrix} and $\Omegab_N$ is the block--diagonal matrix whose diagonal blocks 
are the coefficients of the finite Fourier transform of the first block row of $\Pi_{\Lambda}$.  
In fact, \eqref{eqn:blk_diag_circ_grad} yields
$$
\Pi_{\Lambda}^{-1} = \Vb \Omegab_N^{-1} \Vb^\ast\,,
$$
so that the cost of computing $\Pi_{\Lambda}^{-1}$ reduces to the cost of singularly inverting the $m \times m$ diagonal blocks of $\Omegab_N$ and indeed,  
by exploiting the Hermitian symmetry of the diagonal blocks of $\Omegab_N$, 
to the cost of inverting only the first $\left\lceil \frac{N+1}{2}\right\rceil$  $m \times m$ blocks of $\Omegab_N$.  
As a final improvement, notice that due to the final left and right multiplication by $E_n^\top$ and $E_n$, 
only the first $n+1$ blocks of $\Pi_{\Lambda}^{-1}$ 
are needed to compute the gradient. 

To recap, the proposed procedure reduces the computational cost of each iteration of a generic first--order descent method to $O(m^3)$ flops, in place of the $O(N^3)$ operations per iteration which would have been required by a straighforward application of duality theory.

In the following, we apply a gradient descent method to the optimization of the modified functional $\bar{J}$. 
The overall proposed procedure  is as follows. 

\begin{algorithm}                      
\caption{Matricial gradient descent algorithm}          
\label{alg:MGD}                           
\begin{algorithmic}                    
\STATE Given a starting point $\Lambda \, \in \, \text{dom} \, \Jb$, $\alpha \in (0,0.5)$, $\beta \in (0,1)$ 
\WHILE {$\left\|\nabla_\Lambda \Jb(\Lambda)\right\|_2 > \eta$} 
\STATE $\Delta \Lambda := - \nabla_\Lambda \Jb(\Lambda) $
\WHILE {$\Jb(\Lambda+t\Delta \Lambda) > \Jb(\Lambda) + \alpha t \,\Tr\left\{\nabla \Jb(\Lambda)^\top \Delta \Lambda\right\}$ }
\STATE $t:=\beta t$ 
\ENDWHILE
\STATE $\Lambda := \Lambda + t \Delta \Lambda$
\ENDWHILE
\end{algorithmic}
\end{algorithm}

In the next subsection we provide an efficient initialization for Algorithm \ref{alg:MGD}. 

A comparison of the proposed procedure with  state of the art algorithms for DME from the literature will be presented in Section \ref{sec:num_exp}.

\subsection{Algorithm initialization}\label{subsec:initialization}

{In this Section we exploit the asymptotic result in Theorem \ref{thm:link_TMEExt_CMEExt} 
to provide a good  starting point for the iterative procedure of Algorithm \ref{alg:MGD}. }
To this aim,  
recall that the  maximum entropy completion of a partially specified block--\emph{Toeplitz} matrix can be computed via the formula   
\begin{equation}\label{eqn:ME_spectrum_Tryphon}
\Phi (z)= \left(G^\ast(z) \Tb_n^{-1} \tilde{B} \left(\tilde{B}^\ast \Tb_n^{-1} \tilde{B}\right)^{-1} \tilde{B}^\ast \Tb_n^{-1} G (z)\right)^{-1} 
\end{equation}
(see \cite{Georgiou-06} for details), where
\begin{equation}\label{eqn:tf_func_G}
G(z)= \left(zI-\tilde{A}\right)^{-1} \tilde{B}
\end{equation}
with 
\begin{equation}\label{eqn:BandA_tilde}
\tilde{B} = \bmat 0 \\0 \\ \vdots \\0 \\ I\emat\,, \qquad \tilde{A}= \bmat 
0 & I & 0 & \dots & 0 \\
0 & 0 & I & \dots & 0 \\
\vdots &  &  & \ddots &   \\
0 &  &  &  & I \\
0 & \dots & \dots & \dots & 0
\emat\,.
\end{equation}
\begin{figure}[t!]
\centering
\includegraphics[width=0.6\textwidth]{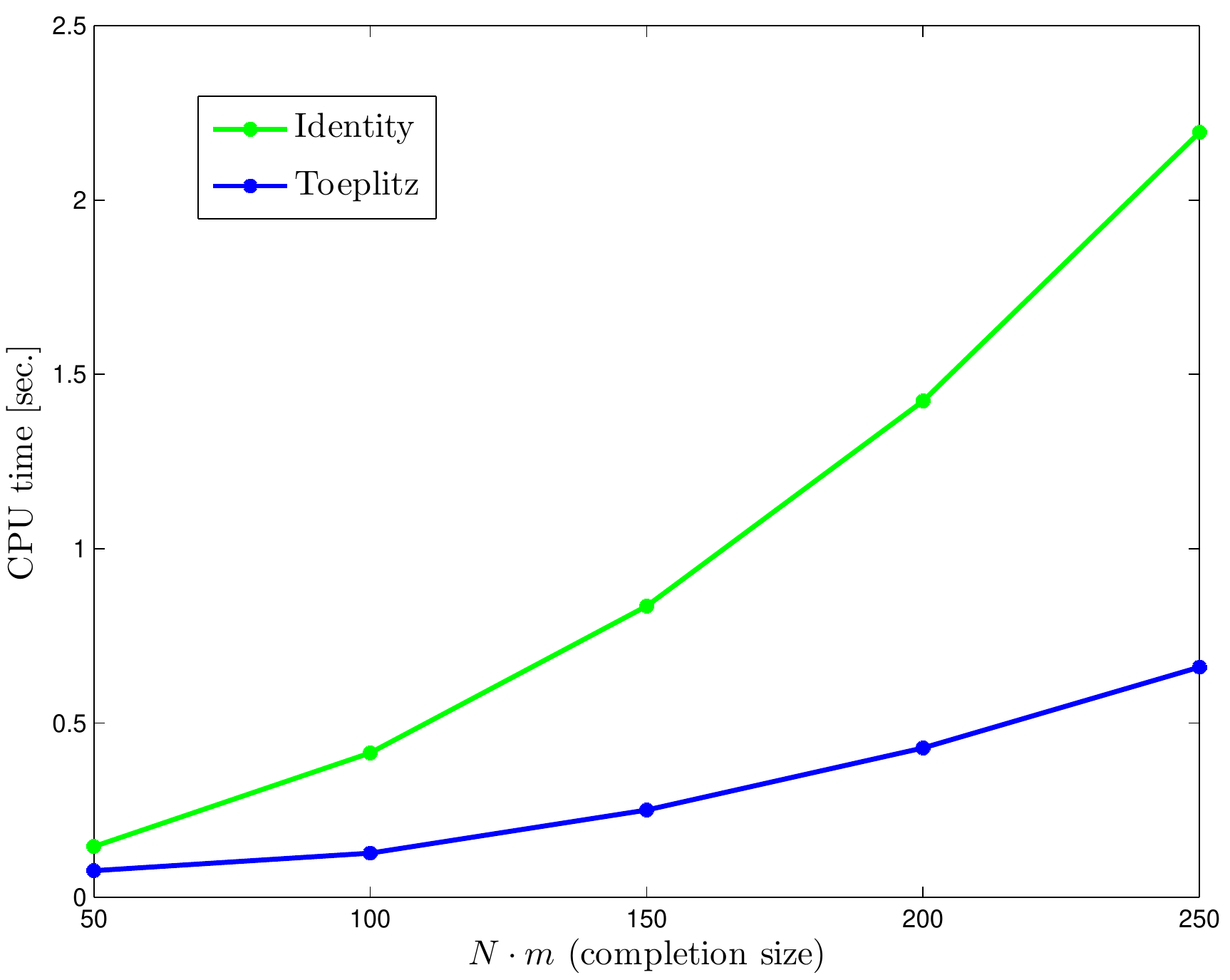}
\caption{CPU time [in sec.] for the matricial gradient descent algorithm with different initializations 
(identity in green and as in Section \ref{subsec:initialization} in blue). 
The reported times have been computed for $N=[10,20,30,40,50]$, $m=5$ and bandwidth $n = 3$.   \label{fig:Cfr_GDIvsGDT_m5n3}}
\end{figure}

\begin{table}[h!]
\begin{center}
\begin{tabular}{cc|cc|cc}
  & & \multicolumn{2}{c}{Identity}  & \multicolumn{2}{c}{Toeplitz}  \\ 
     $N$ & $m$ &     \# of itz.  & CPU time & \# of itz. & CPU time \\
 \hline         
10 & 5 & 99   & 0.1455 &  61  &  0.0767\\
20 & 5 & 212  &  0.4143 &  65 &   0.1270\\
30 & 5 & 322  &  0.8355 &  97 &   0.2504\\
40 & 5 & 432  &  1.4233 & 130 &   0.4285\\
50 & 5 & 541  &  2.1937 & 163 &   0.6603
\end{tabular}
\end{center}
\caption[]{CPU time [in sec.] for the matricial gradient descent algorithm with different initializations 
(identity on the left and as in Section \ref{subsec:initialization} on the right). 
The reported times have been computed for $N=[10,20,30,40,50]$, $m=5$ and bandwidth $n = 3$.   }
\label{tab:Cfr_GDIvsGDT_m5n3}
\end{table}
It follows that the spectral factor $W(z):=\left[L_n(z^{-1})\right]^{-1} \Lambda_n^{\frac12}$ has a realization
$$
W(z)= C(zI- A)^{-1} B + D
$$
with $D = \Lambda_n^{\frac12}$, $C= -\bmat A_n(n) & A_n(n-1) & \dots & \dots & A_n(1) \emat$ and
$$
A = \bmat 0 & I_m & 0 & 0 & \dots & 0 \\
0 & 0 & I_m & 0 & \dots & 0 \\
\vdots & & & \ddots &  & \vdots \\
\vdots & & & & \ddots & 0 \\
0 & \dots & \dots & \dots & 0 & I_m \\
-A_n(n) & -A_n(n-1) & \dots & \dots & \dots & -A_n(1) 
\emat \,, \qquad 
B = \bmat 0 \\ 0 \\ \vdots \\ 0 \\ \Lambda_n^{\frac12}\emat \,,
$$
The positive real part of the maximum entropy spectrum is given by 
\begin{equation}\label{eqn:Phi_+}
\Phi_{+}(z) = C(zI-A)^{-1} \bar{C}^\top + \frac12 \, \Sigma_0
\end{equation}
where $\bar{C} ^ \top = APC^\top + BD^\top$, 
with $P = APA^\top + BB^\top$
and the maximum entropy covariance extension results
$$
\hat{\Sigma}_k = C A^{k-1} \bar{C}^\top \,, \qquad k>n .
$$
With this extension at hand, we can compute an approximation for the maximum entropy block--circulant extension as suggested by 
Theorem \ref{thm:link_TMEExt_CMEExt}. 
A good starting point for our 
algorithm can then be obtained from \eqref{subeq:Pi_k} assuming for $\Lambda$ a Toeplitz structure.  

{As an example, we have compared the execution time of the proposed 
algorithm initialized with the identity matrix and initialized 
with the solution of the associated matrix extension problem for Toeplitz matrices as described above
for 
blocks of size $m=5$, bandwidth $n=3$ and $N$ varying between $10$ and $50$. 
The computational times 
are reported in Figure \ref{fig:Cfr_GDIvsGDT_m5n3} along with Table \ref{tab:Cfr_GDIvsGDT_m5n3}. 
The simulation results confirm that} the proposed initialization acts effectively to reduce the number of iterations (and thus the computational time) required to reach the minimum. 

\section{Algorithms for the unstructured covariance selection problem}\label{sec:alg_cov_sel}
{In this Section we introduce and discuss some of the main algorithms in the literature for the 
positive definite matrix completion problem with the aim of comparison with our newly proposed algorithm.  }

In the literature concerning matrix completion problems,     
it is   common practice  
to describe the pattern of the specified entries of an $mN\times mN$ partial symmetric matrix $M=(m_{ij})$ by an undirected graph 
of $mN$ vertices which has an edge joining vertex $i$ and vertex $j$ if and only if the entry $m_{ij}$ is specified. 
Since the diagonal entries are all assumed to be specified, we ignore loops at the vertices.
\newsavebox{\tempbox}
\begin{figure}[htbp]%
\centering 
\sbox{\tempbox}{\includegraphics[width=0.39\textwidth]{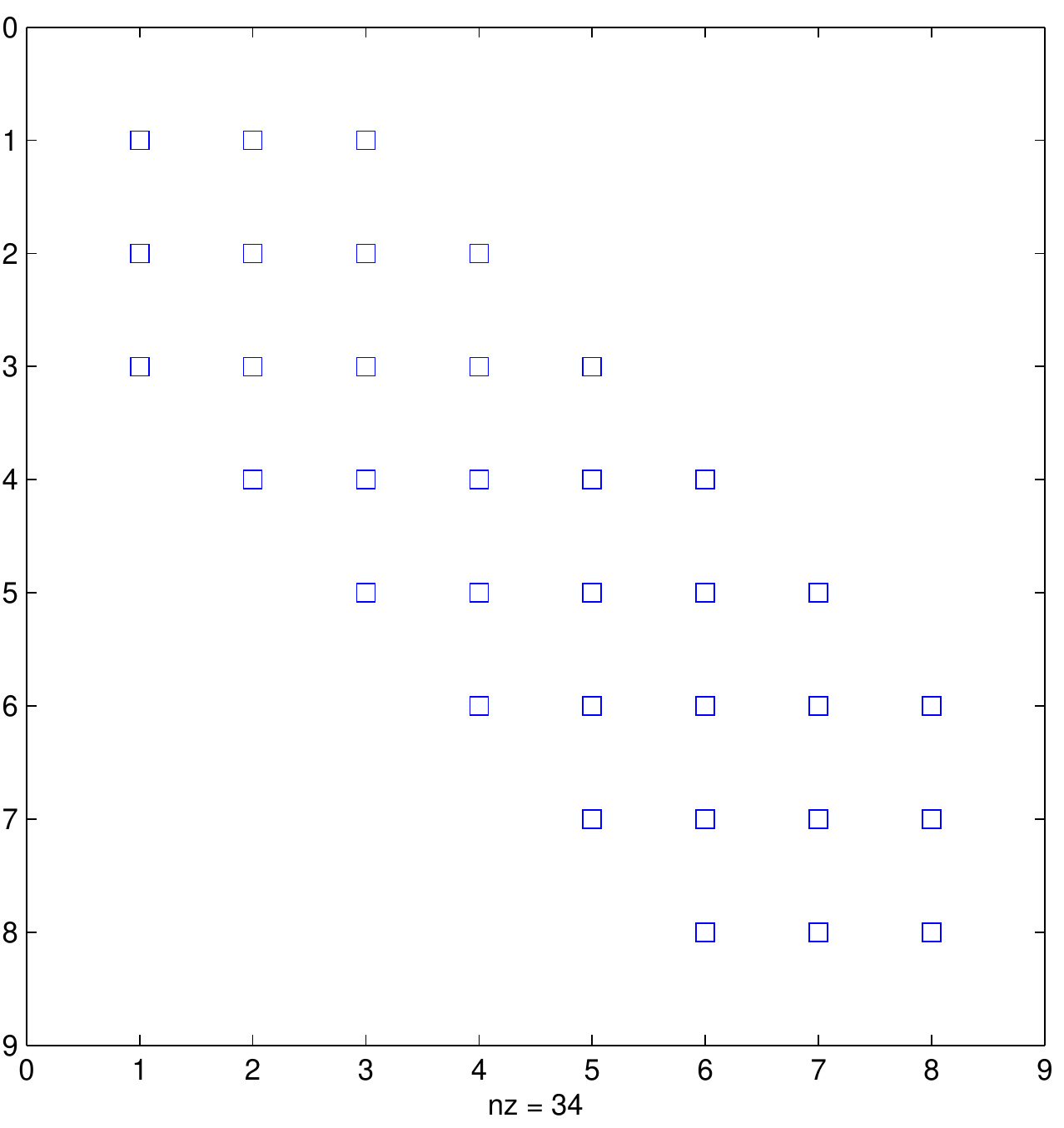}}
\subfloat[][]{\usebox{\tempbox}\label{fig:ToeplBandedSparsityPattern_N8m2}}%
\qquad
\subfloat[][]{\vbox to \ht\tempbox{%
  \vfil
  \hbox{\includegraphics[width=0.39\textwidth]{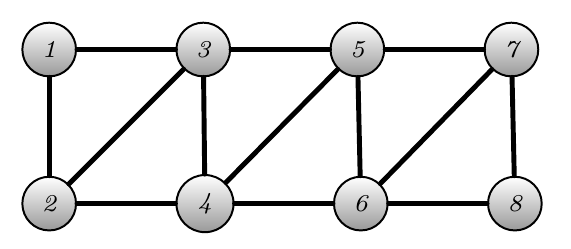}\label{fig:TSP_N8_mod}}
  \vfil}}%
  \caption{Banded pattern of the given entries for the TME problem 
with $N=8$, $n=2$, $m=1$ (on the left) and associated graph (on the right).\label{fig:es_grafo_N10_n8}}
\end{figure}
\begin{figure}[htbp]
\centering 
\subfloat[][] {\includegraphics[width=0.39\textwidth]{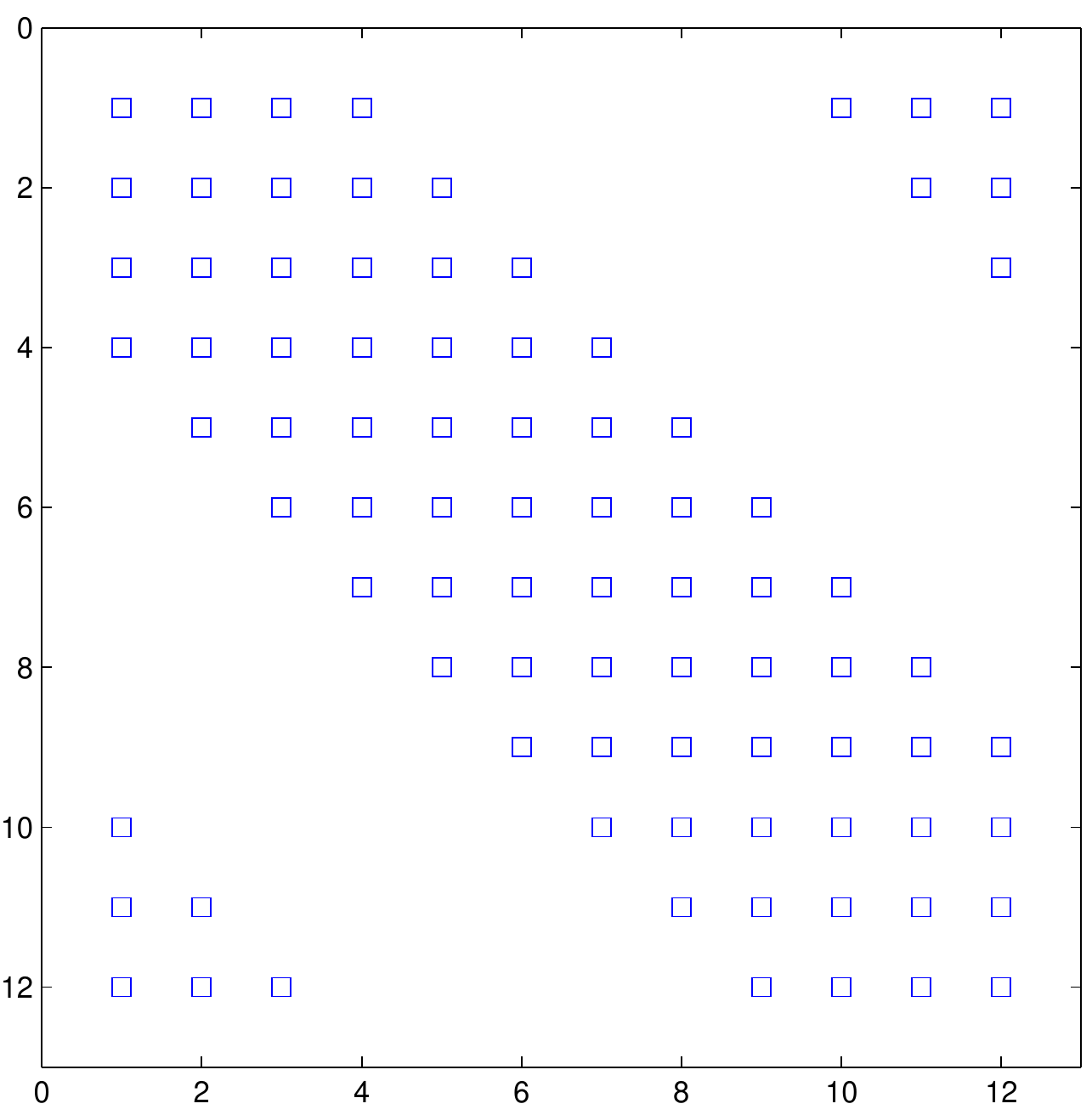}\label{fig:SparsityPatternN12n3}}\hspace{1.5cm}
\subfloat[][] {\includegraphics[width=0.39\textwidth]{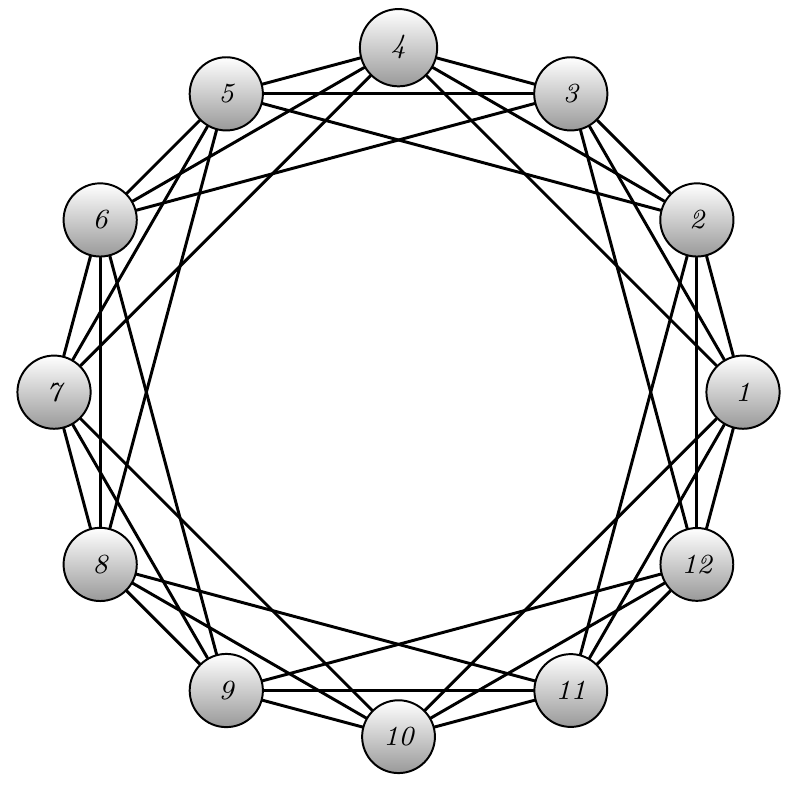}\label{fig:grafo_N12_n3}}
\caption{Banded pattern of 
the given entries for the CME 
with $N=12$, $n=3$, $m=1$ (on the left) and associated graph (on the right). 
The graph is \emph{not} chordal since, for example, the cycle $\left\{1,4,7,10\right\}$ does not have a chord. \label{fig:es_grafo_N12_n3}\vspace{-5mm}}
\end{figure}

As anticipated in the Introduction, if the graph of the specified entries is \emph{chordal} (i.e., a graph in which every cycle of length greater than three has an edge connecting nonconsecutive nodes, see e.g.\cite{Golumbic-80}), 
the maximum determinant matrix completion problem admits a closed form solution in terms of the principal minors of the sample covariance matrix 
(see \cite{BJL-89}, \cite{FKMN-00}, \cite{NFFKM-03}).  
{An example of chordal sparsity pattern along with the associated graph is shown in Figure \ref{fig:es_grafo_N10_n8}.} 
However, the graph associated with a banded circulant sparsity patterns is not chordal, as it is apparent from 
the example of Figure \ref{fig:grafo_N12_n3}.
Therefore we have to resort to \emph{iterative algorithms}. 
For the applications we have in mind, we are dealing with vector--valued processes possibly defined on a quite large interval.  
A straightforward application of standard optimization algorithms is too expensive for problems of such a size, 
and 
anumber of specialized algorithms have been proposed in the graphical models literature (\cite{Dempster-72,SK-86,Wermuth-S-77,Kullback-68}).  
In his early work (\cite{Dempster-72}), Dempster himself proposed two iterative algorithms which however are very demanding from a computational point of view. 
Two popular methods are those proposed by T. P. Speed and H. T. Kiiveri in \cite{SK-86}, that we now briefly discuss. 

\paragraph{Speed and Kiiveri's algorithms}  
We will denote an undirected graph by $\Gc=(V, E)$, where $V$ is the vertex set and $E$ the edge set which consists of 
unordered pairs of distinct vertices. 
In any undirected graph we say that $2$ vertices $u$, $v \in V$ are \emph{adjacent} if $(u,v)\in E$. 
For any vertex set $S \subseteq V$, consider the edge set $E(S)\subseteq E$ given by 
$$
E(S):=\left\{(u,v) \in E \mid u,v \in S \right\}
$$
The graph $\Gc(S)=\left(S, E(S)\right)$ is called \emph{subgraph of $\Gc$ induced by $S$}. 
An induced subgraph $\Gc(S)$ is \emph{complete} if the vertices in $S$ are pairwise adjacent in $\Gc$. 
A \emph{clique} is a complete subgraph that is not contained within another complete subgraph.
Finally, we define the complementary graph of $\Gc = (V,E)$ as the graph $\tilde{\Gc}$ with vertex set $V$ and edge set $\tilde{E}$ 
with the property that $(u,v) \in \tilde{E}$ if and only if $u \neq v$ and $(u,v) \notin E$.  

Let $\Ic_b$ be the set of pairs of indices consistent with a banded, symmetric block--circulant structure of bandwidth $n$, i.e. the set of the $(i,j)$'s which satisfies the following rules set
\begin{equation}
\begin{array}{l}
\text{for } i \in \{ 1, \dots , m \}\,, j \in \{i, \dots, mN \}\,, \; 
\text{if } |i-j| \leq m(n+1)-i \Rightarrow (i,j) \in \Ic_b \nonumber\\
\text{if } (i,j) \in \Ic_b \Rightarrow \Big((i+m)_{\text{mod mN}},(j+m)_{\text{mod mN}}\Big) \in \mathcal{I}_b\nonumber \\
\text{if } (i,j) \in \Ic_b \Rightarrow (j,i) \in \Ic_b \nonumber\\
\end{array}
\end{equation} 
(an example of this structure is shown in Figure \ref{fig:SparsityPatternN12n3}). Moreover, we will denote by 
$\Ic_b^C$ the complement of $\Ic_b$ with respect to $\{ 1, \dots, mN \} \times \{ 1, \dots, mN \}$ and by 
$\Gc = \left( \{1, \dots, mN \}, \Ic_b \right)$ the graph associated with the given entries.

As mentioned in the Introduction, for the class of problems studied by Dempster, 
the inverse of the unique completion which maximizes the entropy functional 
has the property to be zero in the complementary positions of those fixed in $\Sigma_N$.   
Thus, a rather natural procedure to compute the solution of the covariance selection problem for block--circulant matrices seems to be the following: 
iterate maintaing the elements of $\Sigmab_N$ 
{indexed by} $\Ic_b$ at the desired value (i.e. equal to the corresponding 
elements in the sample covariance matrix) while forcing the elements of $\Sigma_N^{-1}$ in $\Ic_b^C$ to zero.  
{To this aim, the following procedure can be devised.}

\begin{algorithm}                      
\caption{First algorithm (Speed and Kiiveri \cite{SK-86})}          
\label{alg:first_alg_SK}                           
\begin{algorithmic}                    
\STATE Compute all the cliques $\ct_t$ in the complementary graph $\tilde{\Gc}$, say $\left\{\ct_t, t=1,\dots, n_{\ct_t}\right\}$; 
\STATE Initialize $\Sigmab_N^{(0)} = \Rb_N$; 
\WHILE {some stopping criterion is satisfied}
\FORALL { the cliques $\ct_t$ in $\tilde{\Gc}$} 
\STATE 
\vspace{-7mm}
\begin{equation*}
		\Sigmab^{(t)}_N = \Sigmab^{(t-1)}_N + \phi\left(\Sigmab^{(t-1)}_N\right)
		\end{equation*}
\vspace{-7mm}		
\ENDFOR
\ENDWHILE
\end{algorithmic}
\end{algorithm}
{
\noindent where $\phi\left(\Sigmab^{(t-1)}_N\right)$ is the $mN \times mN$ zero matrix which equals  
\begin{equation*}
\left\{\text{diag}\left[\left((\Sigmab^{(t-1)}_N)^{-1}\right)_{\ct_t}\right]^{-1}\right\}^{-1} 
- \left[\left((\Sigmab^{(t-1)}_N)^{-1}\right)_{\ct_t}\right]^{-1} 
\vspace{2mm}
\end{equation*}
in the positions corresponding to the current clique $\ct_t$   
(given a $mN \times mN$ matrix $M$ and a set $a \subset \{1,\dots, Nm\}$, $M_{a}$ denotes the submatrix with entries $\left\{m_{ij} \, : \, i,\, j \in a\right\}$). 
Every cycle consists of as many steps as 
the cliques in the complementary graph $\tilde{\Gc}$ (the graph associated to the elements 
indexed by $\Ic_b^C$).    
At each step, only the elements in $\Sigmab_N$ corresponding to the current clique $\ct_t$ (i.e. only a subset of the entries 
indexed by $\Ic_b^C$) are modified in such a way to set the elements of $\Sigmab_N^{-1}$ in the corresponding positions to the desired zero--value. 
Throughout the iterations, the elements in  $\Sigmab_N$ 
are fixed over $\Ic_b$, while the elements of $\left(\Sigmab_N \right)^{-1}$ 
vary over $\Ic^C_b$. 

The role of $\Sigma_N$ and $\Sigma_N^{-1}$ can also be swapped, yielding an alternative procedure, 
which is the analog  of iterative proportional scaling (IPS) for contingency tables \cite{Haberman-74}. 
Let $\varphi\left(\Sigmab_N^{(t-1)}\right)$ be the $mN \times mN$ zero matrix which equals 
\begin{equation}\label{eqn:SK2_2}
\left((\Rb_N)_{c_t}\right)^{-1}-\left(\left(\Sigmab_N^{(t-1)}\right)_{c_t}\right)^{-1}
\end{equation}
in the positions corresponding to the current clique $c_t$ in $\Gc$ (the graph associated with the given entries).  
The second algorithm reads as follows. }

\begin{algorithm}[ht]                     
\caption{Second algorithm (Speed and Kiiveri \cite{SK-86})}          
\label{alg:second_alg_SK}                           
\begin{algorithmic}                    
\STATE Compute all the cliques $c_t$ in ${\Gc}$, say $\left\{c_t, t=1,\dots, n_{c_t}\right\}$; 
\STATE Initialize $\Sigmab_N^{(0)} = I_{mN}$; 
\WHILE {some stopping criterion is satisfied}
\FORALL {the cliques $c_t$ in $\Gc$} 
\STATE 
\vspace{-5mm}
\begin{equation}\label{eqn:SK2_1}
\left(\Sigmab_N^{(t)}\right)^{-1} = \left(\Sigmab_N^{(t-1)}\right)^{-1} + \varphi\left(\Sigmab_N^{(t-1)}\right)
\end{equation}
\ENDFOR
\ENDWHILE
\end{algorithmic}
\end{algorithm}
\vspace{-4mm}

Every cycle consists of as many steps as the cliques in the graph of the specified entries ${\Gc}$.
At each step, only the elements in $\Sigmab_N^{-1}$ corresponding to the current clique $c_t$ 
(i.e. only a subset of the entries 
indexed by $\Ic_b$) are modified in such a way to set the elements of $\Sigmab_N$ 
in the corresponding positions to the desired value, namely equal to the sample covariance $\Rb_N$. 
Through the iterations the elements in  $\left(\Sigmab_N\right)^{-1}$ 
are fixed over $\Ic_b^C$ while the elements of $\Sigmab_N$ 
vary over $\Ic_b$. 
  
The choice of which algorithm is to be preferred depends on the  application and is very much dependent on the number and size of the cliques in $\Gc$ and $\tilde{\Gc}$. In our setting, the complexity of the graph associated with the given entries depends on the bandwidth $n$. 
In particular, for a bandwidth $n$ not too large with respect to the completion size (which is the case we are interested in) the complexity of the graph associated with the given data $\Gc$ is far lower than the complexity of its complementary 
(which, for small $n$, is almost complete), see Figure \ref{fig:grafo_e_grafo_compl_N20m1n258}.
The execution time of the two algorithms has been compared for a completion size $N=30$ and a bandwidth $n$ varying between $2$ and $8$. 
The results are shown in Figure \ref{fig:SK1vs2_nvar_fig} and Table \ref{tab:SK1vs2_nvar_tab}. 
It turns out that for $n$ small the second algorithm (which, from now on, will be referred to as IPS) runs faster than the first, and thus has to be preferred.  

\begin{figure}[htbp!]
\centering 
\subfloat[][$\Gc$ for $n=2$] {\includegraphics[width=0.35\textwidth]{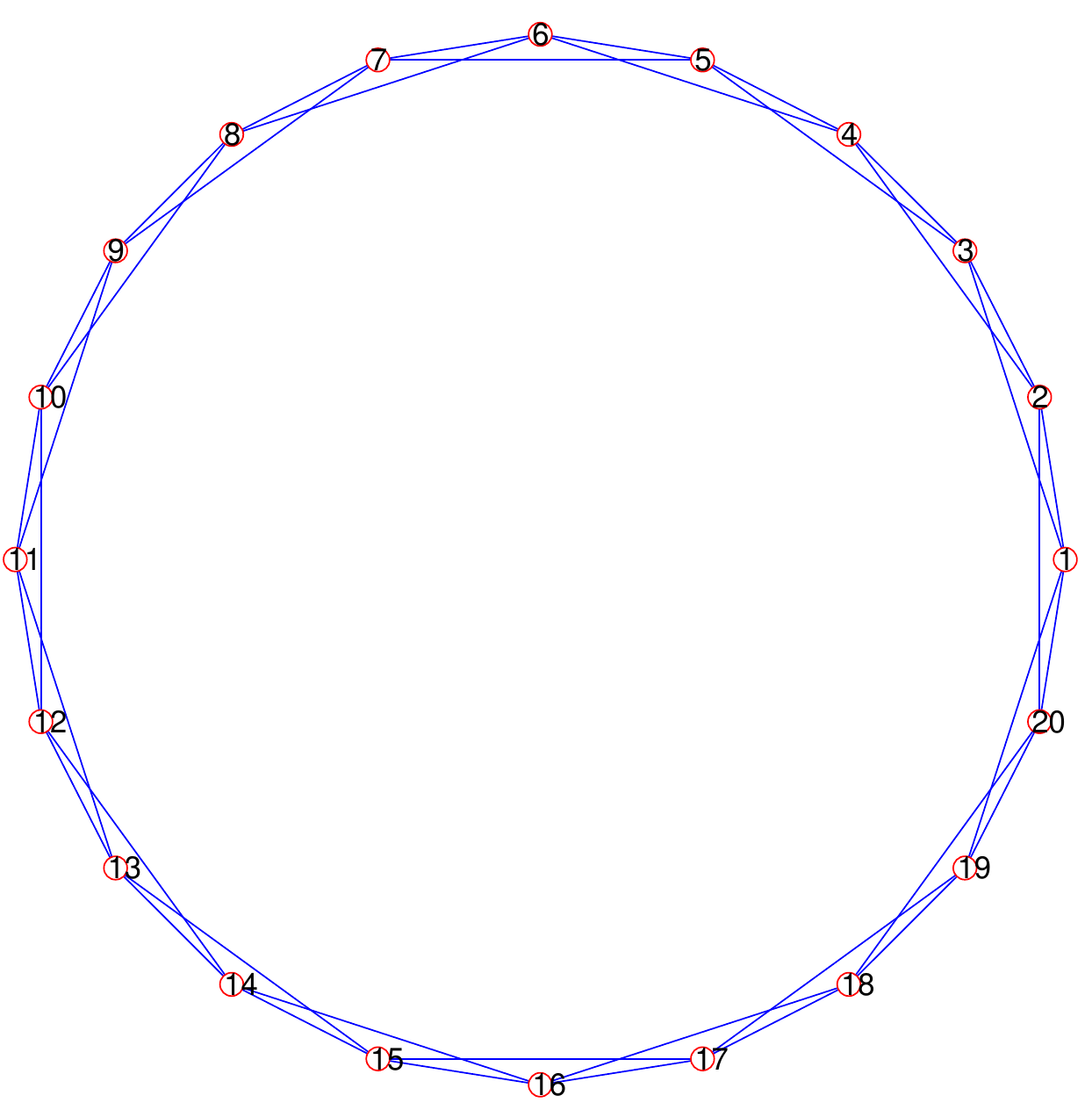}\label{fig:grafo_SK2_N20m1n2}}\hspace{1.5cm}
\subfloat[][$\tilde{\Gc}$ for $n=2$] {\includegraphics[width=0.35\textwidth]{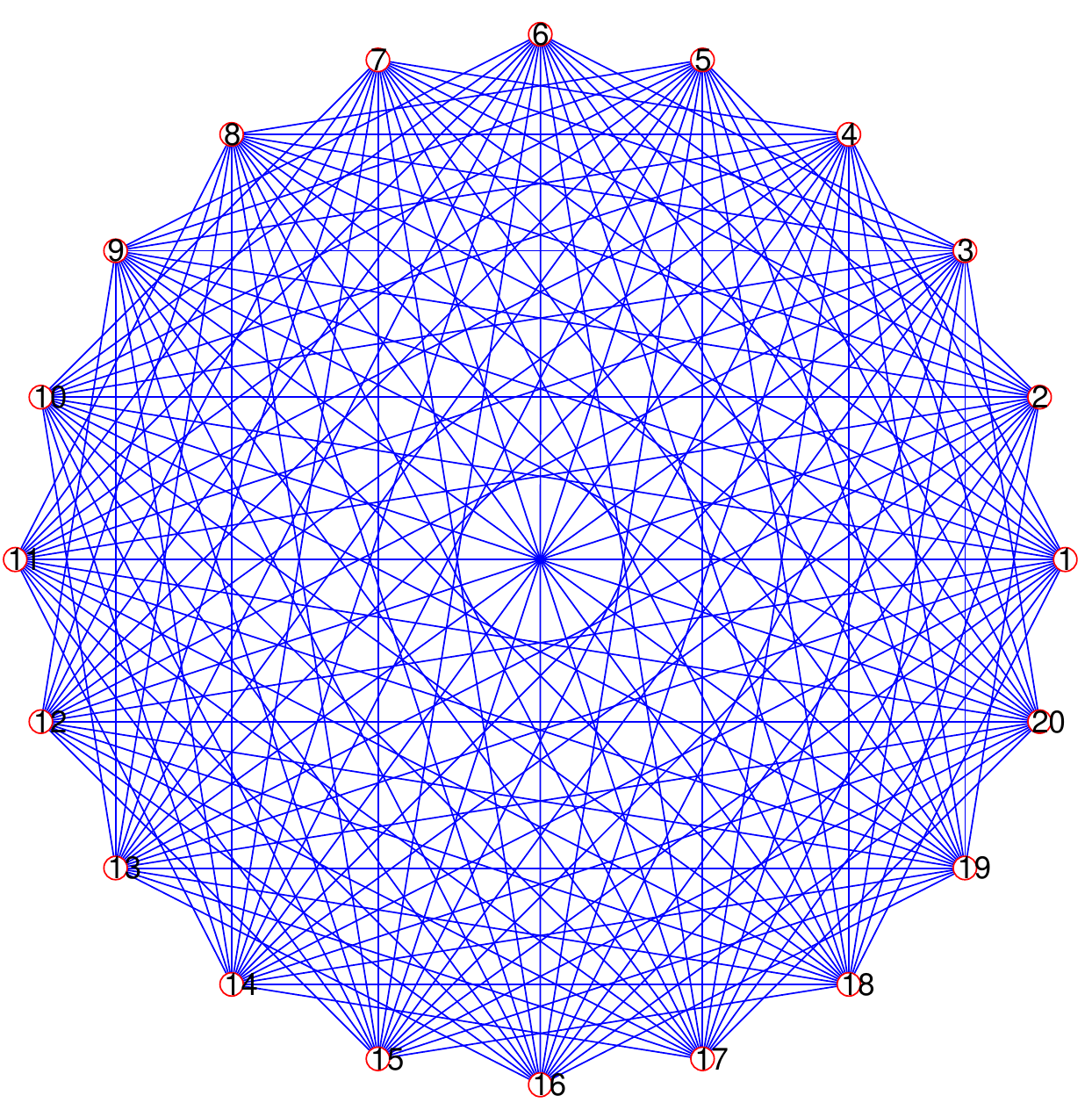}\label{fig:grafo_compl_SK1_N20m1n2}}\\
\subfloat[][$\Gc$ for $n=5$] {\includegraphics[width=0.35\textwidth]{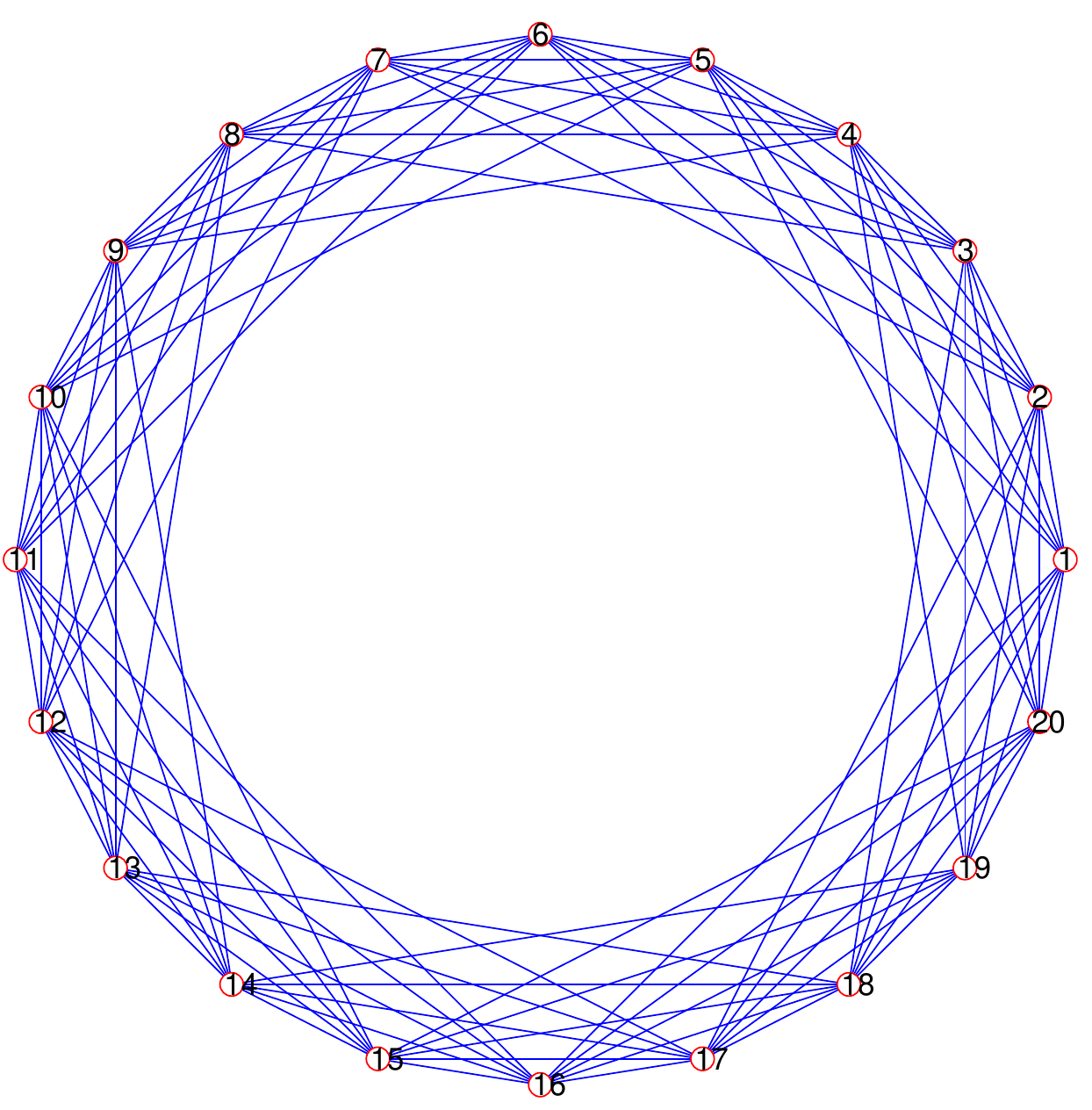}\label{fig:grafo_SK2_N20m1n5}}\hspace{1.5cm}
\subfloat[][$\tilde{\Gc}$ for $n=5$] {\includegraphics[width=0.35\textwidth]{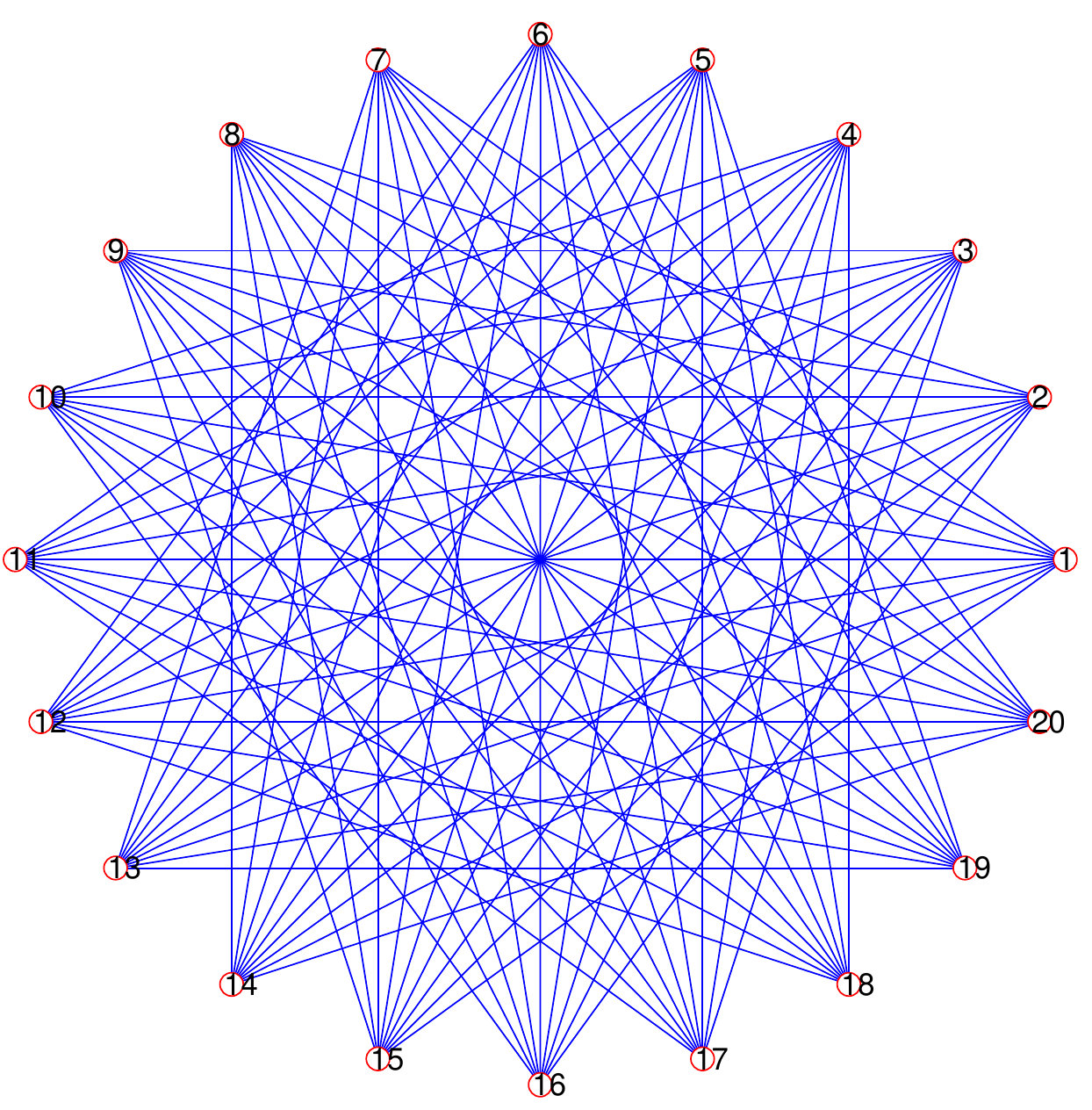}\label{fig:grafo_compl_SK1_N20m1n5}}\\
\subfloat[][$\Gc$ for $n=8$] {\includegraphics[width=0.35\textwidth]{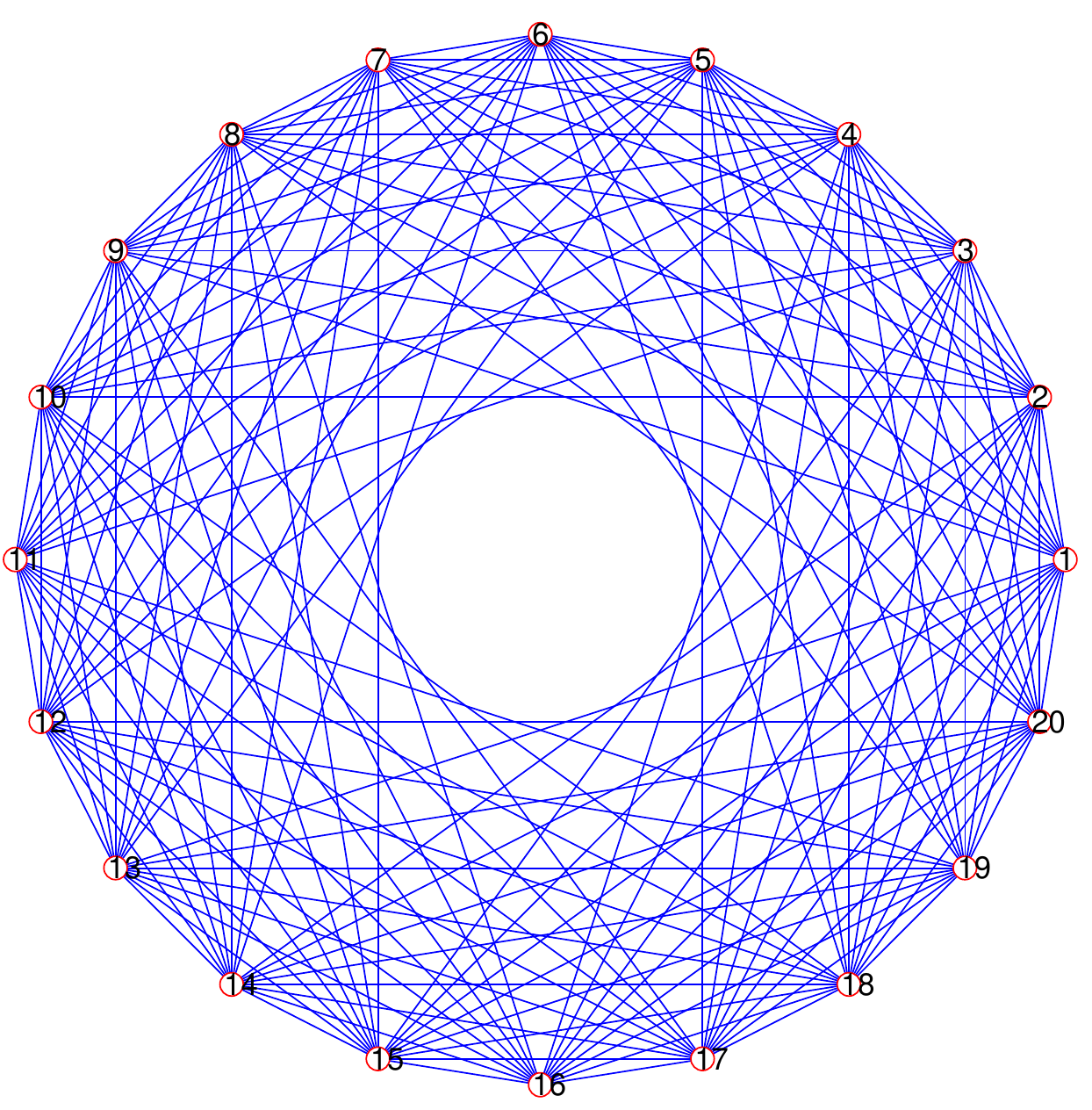}\label{fig:grafo_SK2_N20m1n8}}\hspace{1.5cm}
\subfloat[][$\tilde{\Gc}$ for $n=8$] {\includegraphics[width=0.35\textwidth]{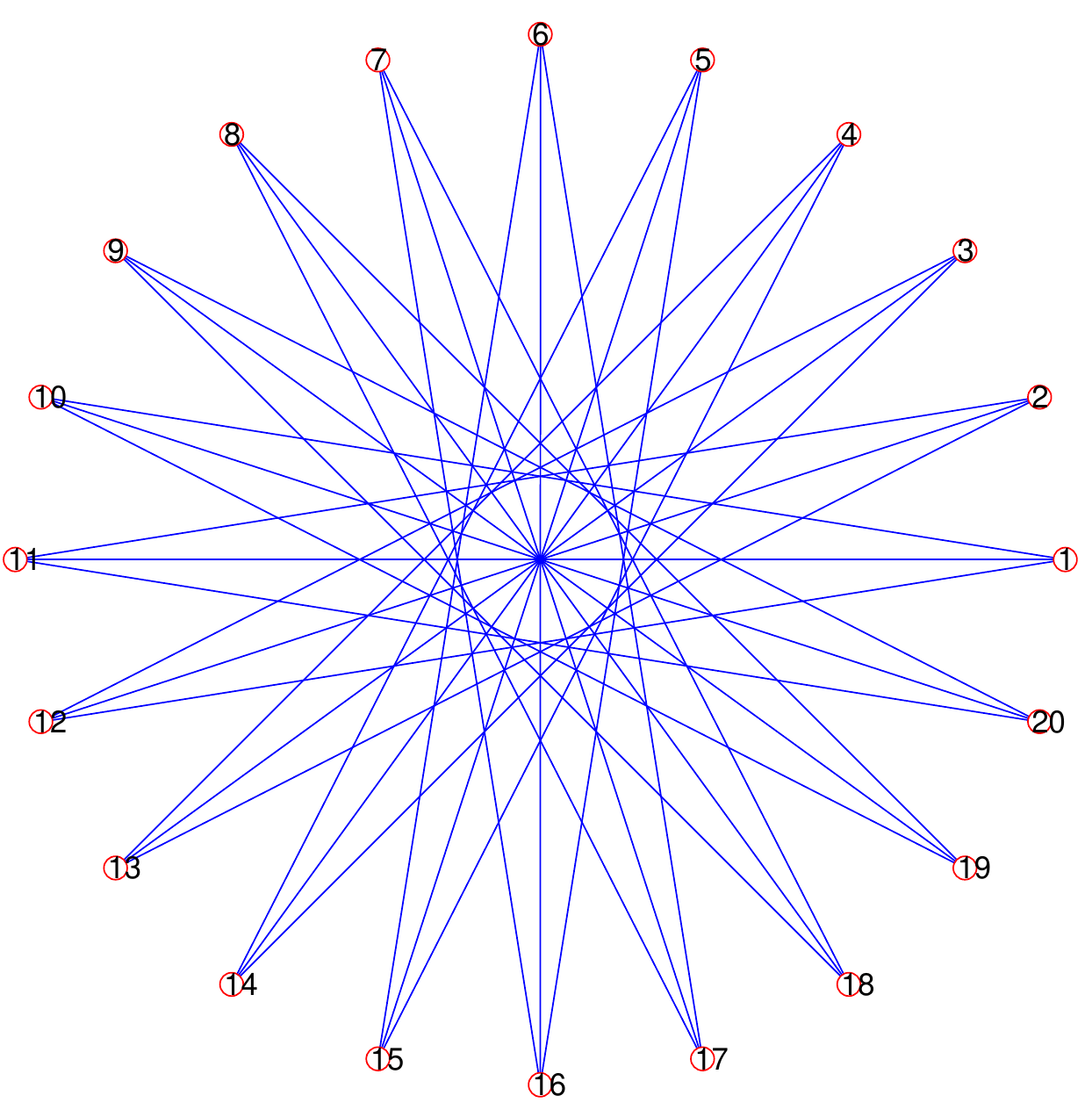}\label{fig:grafo_compl_SK1_N20m1n8}}\\
\caption{Graph ${\mathcal{G}}$ associated with the given data (on the right) and its complementary $\tilde{{\mathcal{G}}}$ (on the left) 
for $N=20$ and bandwidth $n=2,5,8$.\label{fig:grafo_e_grafo_compl_N20m1n258}}
\end{figure}

\begin{figure}[htbp!]
\centering \includegraphics[width=0.6\textwidth]{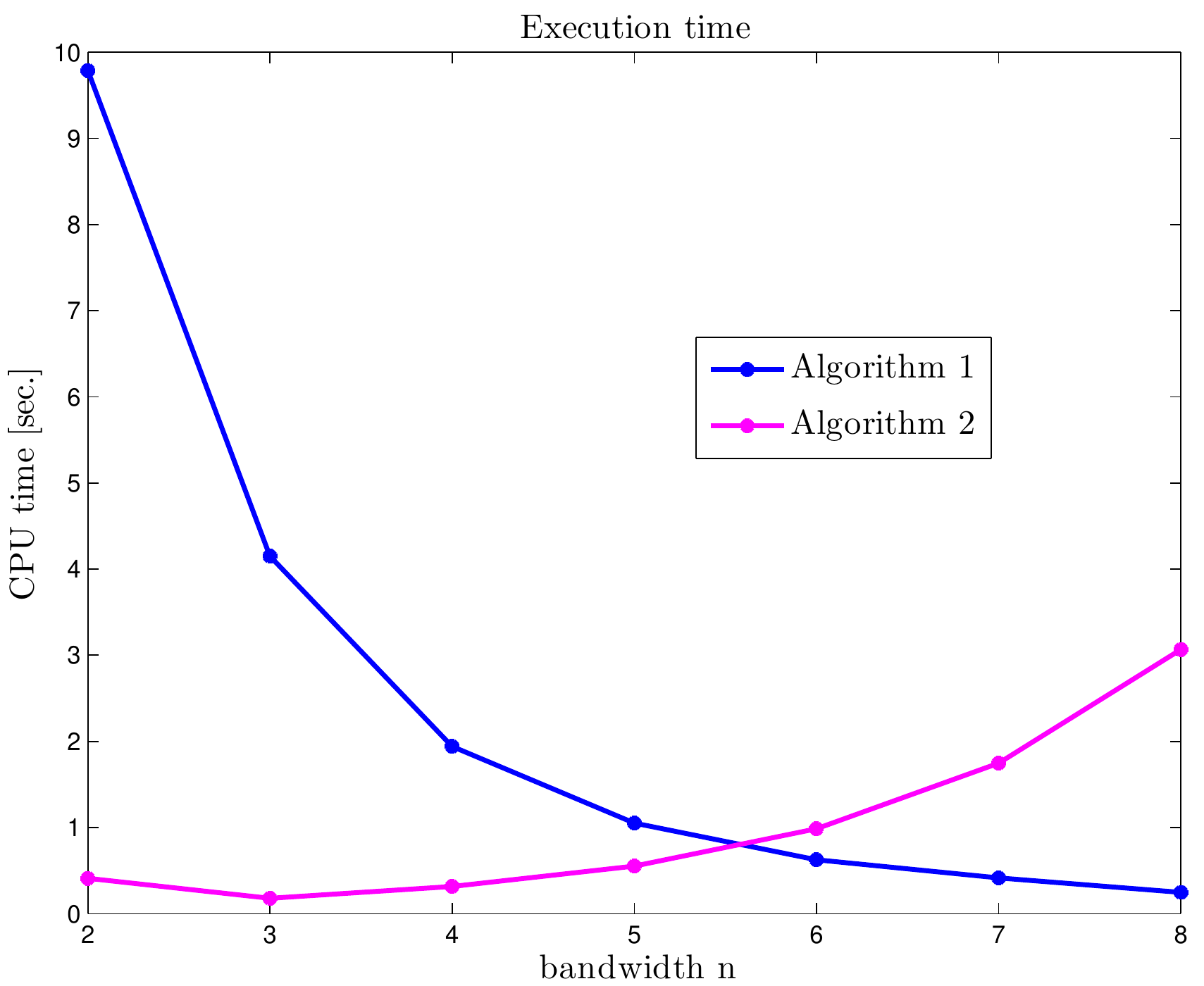}
\caption{\label{fig:SK1vs2_nvar_fig} Comparison between the execution time of the first and second algorithm for 
$N=30$, $m=1$, $n = \left\{2,\dots,8\right\}$. }
\end{figure}

\begin{table}[htbp!]
\begin{center}
\footnotesize{
\begin{tabular}{|c|c|c|c|c|}
\hline & \multicolumn{2}{|c|}{First algorithm}& \multicolumn{2}{|c|}{Second algorithm}\\
\hline
 $n$ &   \# of cl. (max. cl. size) & CPU time [s] & \# of cl. (max. cl. size) & CPU time [s]  \\
\hline \hline
$2$ & $4608(10)$ & $9.7877$  & $30(3)$  & $0.4109$  \\ 
\hline
$3$ & $2406(7)$ & $4.1515$  & $30(4)$  &  $ 0.1783$ \\ 
\hline
$4$ & $1241(6)$ & $1.9419$  & $30(5)$  &  $0.3153$ \\ 
\hline
$5$ & $706(5)$  & $1.0525$  & $30(6)$  &  $0.5535$ \\ 
\hline
$6$ & $445(4)$  & $0.6258$  & $30(7)$  &  $0.9854$ \\ 
\hline
$7$ & $295(3)$  & $0.4145$  & $30(8)$  &  $1.7477$ \\ 
\hline
$8$ & $175(3)$  & $0.2480$  & $30(9)$ &  $3.0665$  \\ 
\hline
\end{tabular}  } 
\end{center}
\caption{Execution time of the first and second algorithm for $N=30$, $m=1$, bandwidth $n = \left\{2, \dots, 8\right\}$. }
\label{tab:SK1vs2_nvar_tab}
\end{table}

\paragraph{Covariance selection via chordal embedding} 
Recently, Dahl, Vanderberghe and Roychowdhury \cite{Dahl-V-R-08} have proposed a new technique to improve the efficiency of the Newton's method for the covariance selection problem based on chordal embedding: the given sparsity pattern is embedded in a chordal one for which they provide efficient techniques for computing the gradient and the Hessian. The complexity of the method is dominated by the cost of forming and solving a system of linear equations in which the number of unknowns depends on the number of nonzero entries added in the chordal  embedding. 
For a circulant sparsity pattern, it is easy to check that the number of nonzero elements added in the chordal embedding is quite large.  
Hence, their method does not seem to be effective for our problem.

\section{Comparison of the proposed algorithm and the IPS algorithm}\label{sec:num_exp}

\begin{figure}[t!]
\begin{minipage}[c]{0.65\textwidth}
\flushleft
\includegraphics[width=0.95\textwidth]{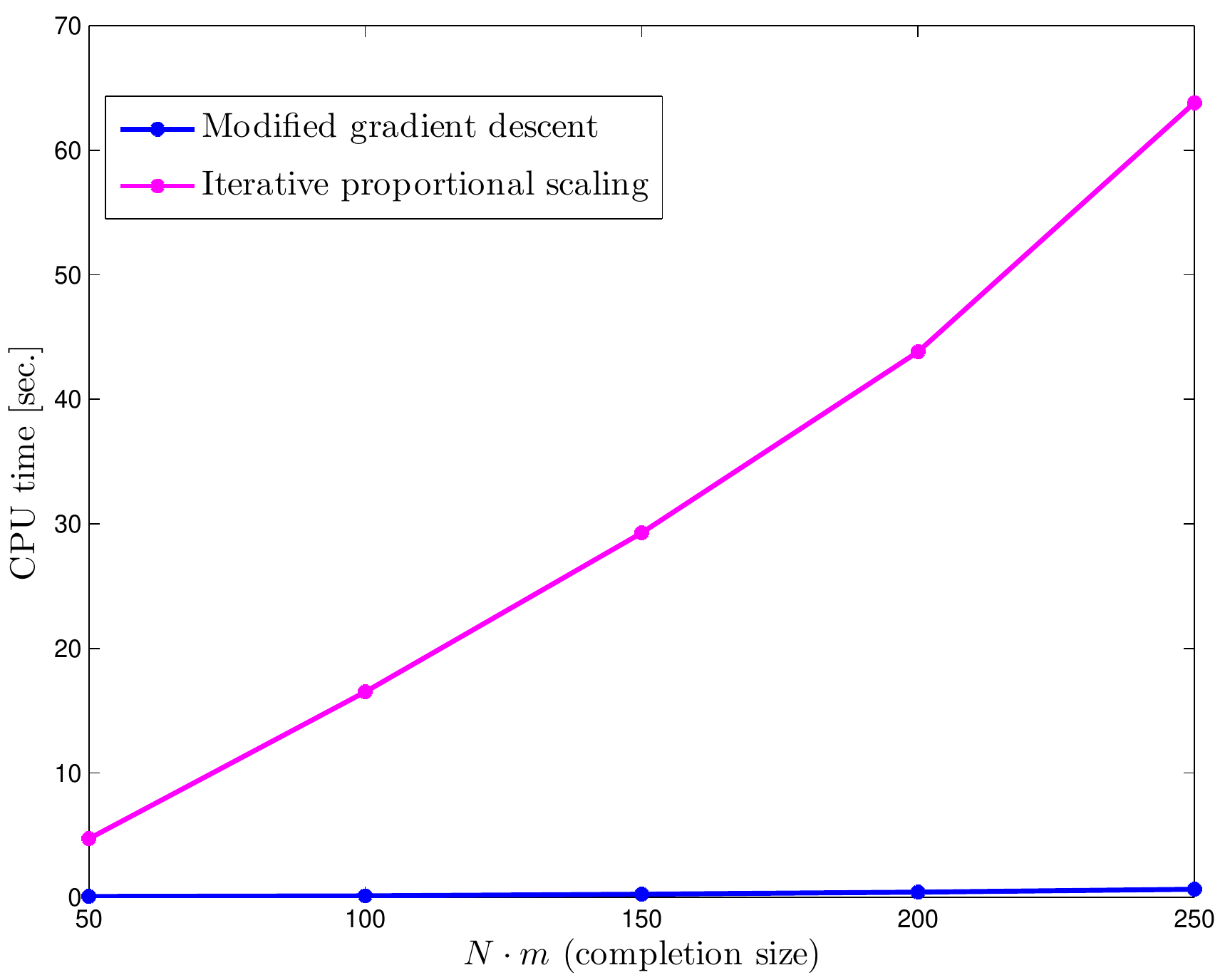}
\end{minipage}
\begin{minipage}[c]{0.05\textwidth}
\flushleft
\footnotesize{
\begin{tabular}{cc|cc}
 $N$ & $m$& IPS  & GD  \\ 
\hline  
10 & 5 & 4.7048   &  0.0767\\
20 & 5 & 16.4981   &  0.1270\\
30 & 5 & 29.2779    &  0.2504\\
40 & 5 &43.8072    &  0.4285\\
50 & 5 &63.8069     &  0.6603
\end{tabular}}
\end{minipage}
\caption{Matricial gradient descent algorithm vs. iterative proportional scaling: CPU time [in sec.] for $N=[10,20,30,40,50]$, $m=5$, 
bandwidth $n = 3$.   \label{fig:Cfr_GDt_vsSK2_Nmin10Nmax50m5n3}}
\end{figure}

The proposed gradient descent algorithm (GD) applied to the modified dual functional $\bar{J}$ has been compared to the iterative proportional scaling procedure (IPS) by Speed and Kiiveri. 
Both algorithms are implemented in Matlab. 
The Bron--Kerbosch algorithm \cite{Bron-Kerbosch-73} has been employed for finding the cliques in the graph for IPS.  
{We recall (see Section \ref{sec:MGD}) that the number of operations per iteration required by our modified gradient descent algorithm is 
cubic in the block--size $m$, as opposed to the $O([m(N-(n+1))]^3)$ operations per iteration of the IPS algorithm (see equations \eqref{eqn:SK2_2} and \eqref{eqn:SK2_1}). 
It follows that for large instances of the CME our newly proposed 
algorithm is expected to run considerably faster than the IPS algorithm. }
The execution times for different completion size $N$ and  block size $m$ are plotted in Figures \ref{fig:Cfr_GDt_vsSK2_Nmin10Nmax50m5n3} and  \ref{fig:Cfr_GDt_vsSK2_Nmin10Nmax50m10n3}. 
The simulation study confirms that our  gradient descent algorithm applied to the modified dual functional $\bar{J}$ 
outperforms the iterative proportional scaling and the gap between the two increases as $N$ increases. Moreover, the gap becomes much more evident as $m$ grows, making the gradient descent algorithm more attractive for applications where the process under observation is vector--valued ($m > 1$). 


\begin{figure}[t!]
\begin{minipage}[c]{0.65\textwidth}
\flushleft
\includegraphics[width=0.95\textwidth]{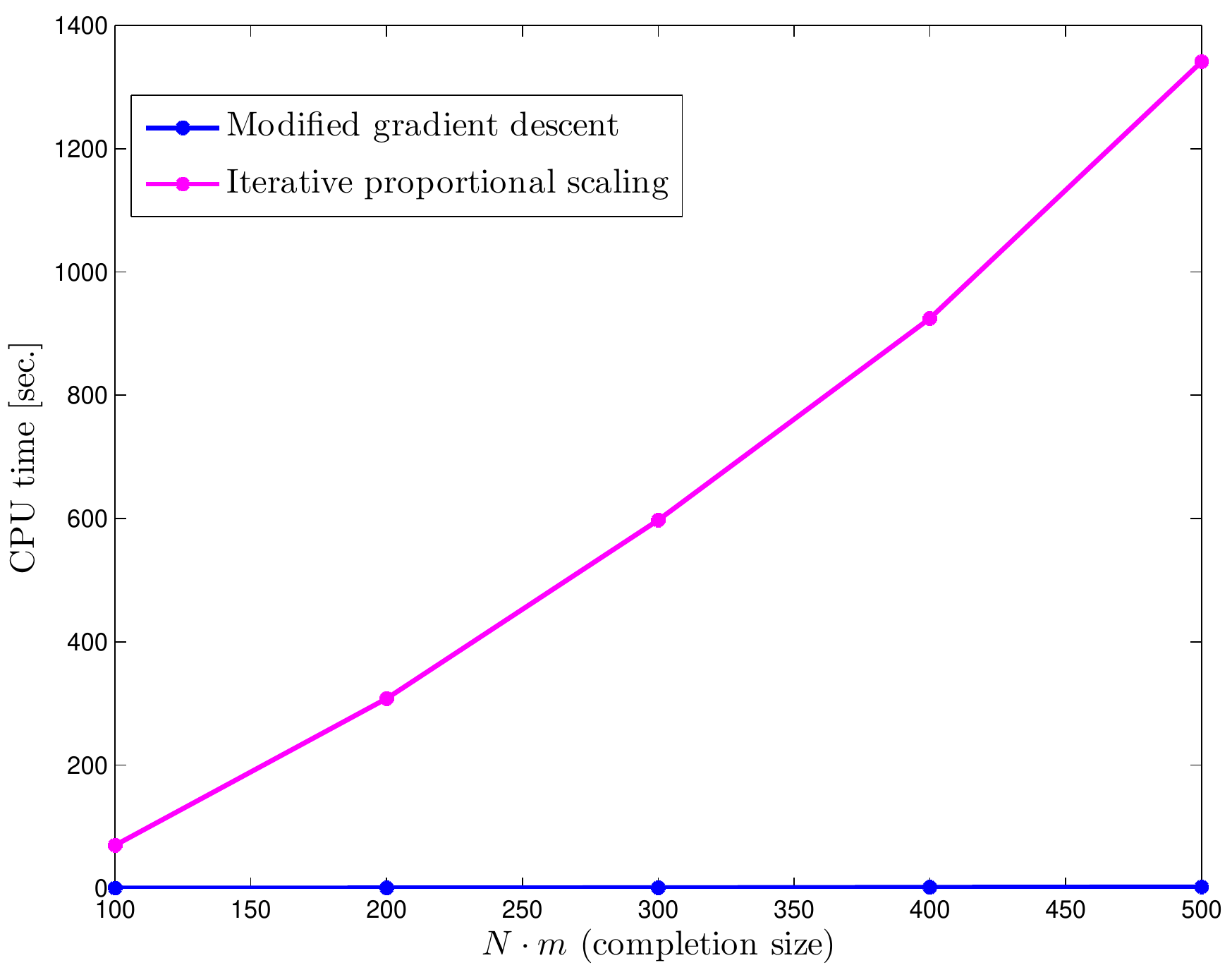}
\end{minipage}
\begin{minipage}[c]{0.05\textwidth}
\flushleft
\footnotesize{
\begin{tabular}{cc|cc}
 $N$ & $m$& IPS  & GD  \\ 
 \hline  
10  & 10 &  69.7671      &  0.1516\\
20  & 10 &  307.9596     &  0.4459\\
30  & 10 &  597.3791     &  0.8988\\
40  & 10 &  924.6431     &  1.4798\\
50  & 10 &  1341.0976    &  2.2052
\end{tabular}}
\end{minipage}
\caption{Matricial gradient descent algorithm vs. iterative proportional scaling: CPU time [in sec.] for $N=[10,20,30,40,50]$, $m=10$, 
bandwidth $n = 3$.    \label{fig:Cfr_GDt_vsSK2_Nmin10Nmax50m10n3}}
\end{figure}

\section{Conclusions} \label{sec:conclusion}
The main contribution of the present paper is an efficient algorithm to solve the maximum entropy band extension problem for block--circulant matrices. 
This problem has many applications in signal processing since it arises in connection with maximum likelihood estimation of periodic, and in particular quasi--Markov 
(or reciprocal), processes. Even if matrix completion problems have gained considerable attention in the past (think for example to the covariance extension problem for stationary processes on the integer line, i.e. for Toeplitz matrices), the maximum entropy band extension problem for block--circulant matrices has been addressed for the first time 
in \cite{CFPP-2011}. The proposed algorithm exploits the 
circulant structure and relies on the variational analysis brought forth in \cite{CFPP-2011}. 
{An efficient initialization for the proposed algorithm is provided thanks to the established relationship between the solutions of the maximum entropy problem for block--circulant and block--Toeplitz matrices. 
Further light is also shed on the feasibility issue for the CME problem.  }

\appendix
\section{Feasibility of the CME: an example}\label{app:example}
In Section \ref{sec:Markov_vs_Rec} 
we have shown that, for given $\sigma_0$ and $\sigma_1$, feasibility of the CME depends on the completion size $N$.
{The following example, aims at clarifying the interplay between feasibility and completion size $N$ 
in the simple case of unitary bandwidth and block--size 
using the characterization of the set of all positive definite completions derived in \cite{CG-2011}. }
\begin{example}
Let $\sigma_0=1$, $\sigma_1 = -0.91$. 
We want to investigate the feasibility of Problem \ref{MaxEntProbl} for $N=7$ and $N=9$,  
{i.e. we want to determine if, for $N=7$ and $N=9$, there exist a positive definite circulant completion for the partially specified matrices 
$$
\Sigmab_7 = {\rm Circ}\left( \sigma_0,\, \sigma_1,\, x,\, y,\,y,\,x,\,\sigma_1\right)\,, \qquad
\Sigmab_9 = {\rm Circ}\left( \sigma_0, \,\sigma_1,\, x, \,y,\,z,\,z,\,y,\,x,\,\sigma_1\right)\,, 
$$
where ${\rm Circ}\left( a \right)$ denotes the circulant symmetric matrix specified by its first row $a$, and $x$, $y$ and $z$ denote the unspecified entries. } 
Since 
$$
\cos \left\{\frac{(N-1)}{N}\pi \right\}=
\begin{cases} 
-0.9010  & \text{for $N=7$} \\
-0.9397  & \text{for $N=9$} 
\end{cases}\,,
$$ 
by Theorem \ref{thm:feas_bs1_bw1}, 
we expect that for $N=7$ the problem is unfeasible while for $N\geq 9$ it is expected to become feasible. 
For $N=7$ the set of all positive definite completions is delimited by the intersection of the half--planes 
indentified by 
{the eigenvalues $\Psi(w^k)$, $k=0, \dots, 6$ of $\Sigmab_7$}    
\begin{align*}
\Psi(w^0) &=-0.82 + 2x + 2y  \\
\Psi(w^1) &= \Psi(w^6) = -0.134751 - 0.445042 x -   1.80194 y \\
\Psi(w^2)&= \Psi(w^5) =1.40499 - 1.80194 x +   1.24698 y  \\
\Psi(w^3)& = \Psi(w^4)= 2.63976 + 1.24698 x -   0.445042 y. 
\end{align*} 
{(see \cite{CG-2011} for details). }  
In Figure \ref{fig:NonFeasN7_semipiani34} the intersection $\Gamma$ of the half--planes 
{$\Psi(w^0)> 0$ and $\Psi(w^1)>0$} 
is shown, 
together with the half--plane {$\Psi(w^2)>0$}. 
The intersection of these two regions is empty.  
It follows that the intersection of the four half--planes {$\Psi(w^k) > 0$}, $k=0, \dots, 3$ is also empty, as claimed.
On the other hand, if $N=9$, the eigenvalues of $\Sigmab_9$ are 
\begin{align*}
\Psi(w^0) &=-0.82 + 2 x + 2 y + 2 z \\
\Psi(w^1) &= \Psi(w^8) = -0.394201 + 0.347296 x - y - 1.87939 z \\
\Psi(w^2)&= \Psi(w^7) =0.68396 - 1.87939 x - y + 1.53209 z\\
\Psi(w^3)& = \Psi(w^6)= 1.91  -  x + 2 y - z \\
\Psi(w^4)& = \Psi(w^5) = 2.71024  + 1.53209 x - y + 0.347296 z 
\end{align*}
and the feasible set is the nonempty region shown in Figure \ref{fig:Feas_reg_N9_3D}. 
\begin{figure}[htbp!]
\centering
\includegraphics[width=0.9\textwidth]{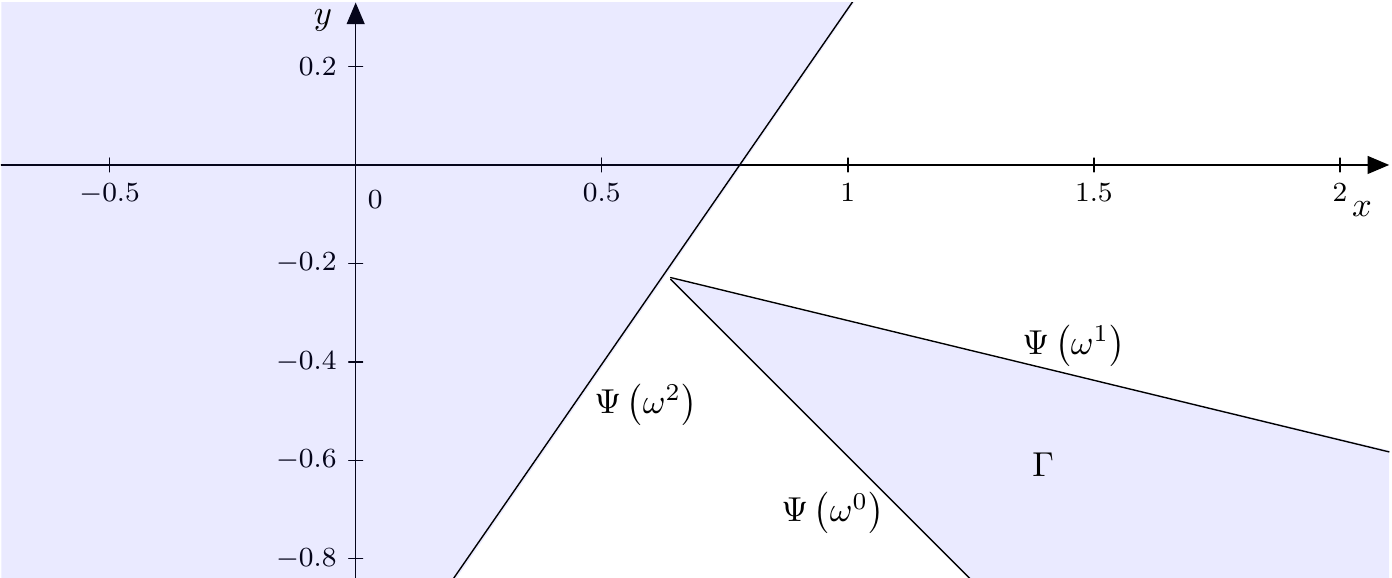}
\caption[]{Half--plane $\Psi(w^2)>0$ and intersection of the half--planes $\Psi(w^0)>0$ and $\Psi(w^1)>0$. The intersection of the two regions is empty.  }
\label{fig:NonFeasN7_semipiani34}
\end{figure}
\begin{figure}[htbp!]
\centering \includegraphics[width=0.55\textwidth]{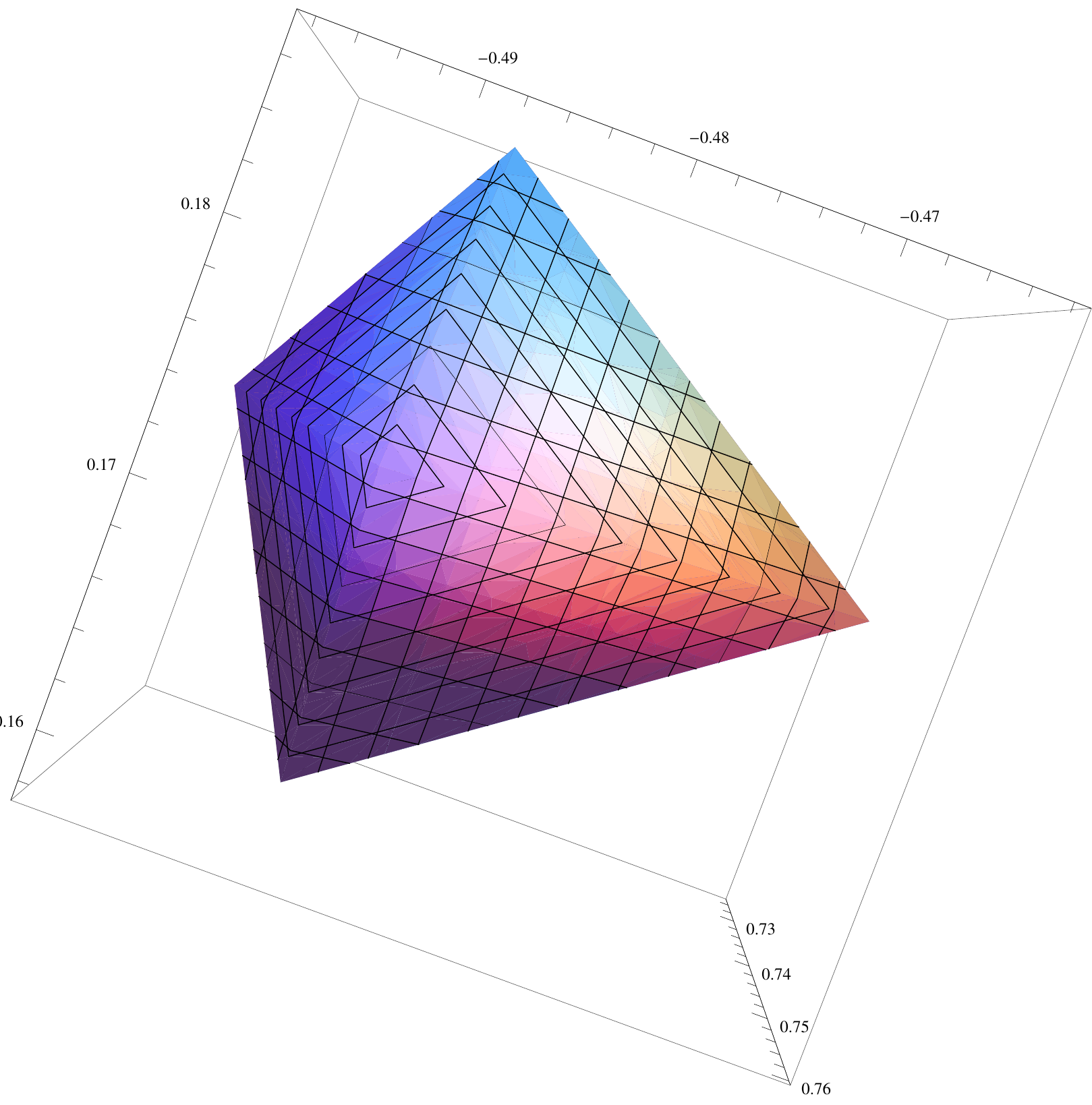}
\caption[Feasible region for $\Sigmab_N=\Circ\left\{1,-0.91,x,y,z,z,y,x,-0.91\right\}$ .]{\label{fig:Feas_reg_N9_3D} Feasible region $\{(x,y,z)\mid \Sigmab_N > 0\}$ for $\Sigmab_N=\Circ\left\{1,-0.91,x,y,z,z,y,x,-0.91\right\}$ .}
\end{figure}
\end{example}





\bibliographystyle{plain}
\bibliography{biblio_algo_v12}







\end{document}